\PassOptionsToPackage{lowtilde}{url}		
\documentclass[11pt]{amsart}
\pdfoutput=1

\usepackage[english]{babel}
\usepackage{amssymb}
\usepackage{xfrac}
\usepackage{amsthm}
\usepackage{mathtools}                          
\usepackage{mathrsfs}				    
\usepackage{verbatim}				    
\usepackage[all]{xy}    	          
\SelectTips{cm}{11}                  
\usepackage[T1]{fontenc}        
\usepackage{graphicx}	
\usepackage[normalem]{ulem}          
\usepackage{microtype}               
\usepackage{lmodern}

\usepackage{tikz}									
\usetikzlibrary{matrix,patterns,knots,calc,arrows,shapes,decorations.pathreplacing}
\usepackage{pbox}

\colorlet{darkblue}{blue!75!black}

\usepackage{enumitem}
\setenumerate[1]{label=(\roman*),ref=\thesubsection.{\roman*}}    

\usepackage{hyperref}
\hypersetup{
  colorlinks   = true,          
  urlcolor     = blue,          
  linkcolor    = blue,          
  citecolor   = red             
}
\newcommand{\A}{{\mathbb{A}}}
\newcommand{\N}{{\mathbb{N}}}
\newcommand{\Z}{{\mathbb{Z}}}
\newcommand{\Q}{{\mathbb{Q}}}
			
\newcommand{\C}{\mathcal{C}}

\newcommand{\PP}{{\mathbb P}}
\newcommand{\pan}{{\mathbb{P}^{1,\mathrm{an}}_K}}	
\newcommand{\aan}{{\mathbb{A}^{1,\mathrm{an}}_K}}	

\newcommand{\calC}{{\mathcal C}}

\newcommand{\calE}{{\mathcal E}}

\newcommand{\calN}{{\mathcal N}}
\newcommand{\calO}{{\mathcal O}}

\newcommand{\frakC}{{\mathfrak C}}
\newcommand{\frakE}{{\mathfrak E}}

\newcommand{\fraks}{{\mathfrak s}}
\newcommand{\s}{\mathfrak s}

\newcommand{\bigslant}[2]{{\raisebox{.2em}{$#1$}\left/\raisebox{-.2em}{$#2$}\right.}}

\newcommand{\X}{\mathscr X}			

\newcommand{\an}{\mathrm{an}}
\newcommand{\pr}{\mathrm{pr}}

\newcommand{\ansk}[1]{{\Sigma^\mathrm{an}(#1)}}
\newcommand{\anskC}{{\Sigma^\mathrm{an}(C)}}
\newcommand{\Kalg}{{\widehat{K^\mathrm{alg}}}}

\DeclareMathOperator{\Spf}{Spf}
\DeclareMathOperator{\Spec}{Spec}

\DeclareMathOperator{\Gal}{Gal}

\DeclareMathOperator{\Sp}{sp}		

\DeclareMathOperator{\lcm}{lcm}

\DeclareMathOperator{\Ends}{Ends}

\newtheoremstyle{plain2}    
  {}            
  {}            
  {\itshape}    
  {}            
  {\bfseries}   
  {.}           
  {5pt plus 1pt minus 1pt}  
  {{\thmnumber{(#2)} \thmname{#1}{\thmnote{ (#3)}}}}          
\theoremstyle{plain2}

\newtheorem{theorem}[subsection]{Theorem}
\newtheorem{corollary}[subsection]{Corollary}
\newtheorem{lemma}[subsection]{Lemma}
\newtheorem{proposition}[subsection]{Proposition}

\newtheoremstyle{definition2}    
  {}   
  {}   
  {\normalfont}  
  {}       
  {\bfseries} 
  {.}        
  {5pt plus 1pt minus 1pt} 
  {{\thmnumber{(#2)} \thmname{#1}{\thmnote{ (#3)}}}}          
\theoremstyle{definition2}

\newtheorem{remark}[subsection]{Remark}
\newtheorem{remarks}[subsection]{Remarks}
\newtheorem{example}[subsection]{Example}
\newtheorem{examples}[subsection]{Examples}

\newtheoremstyle{stepstyle}
  {}  
  {}   
  {\normalfont}  
  {\parindent}       
  {\itshape} 
  {.}         
  {5pt plus 1pt minus 1pt} 
  {{\thmname{#1} \thmnumber{#2}{\thmnote{: #3}}}}          
\theoremstyle{stepstyle}

\newtheoremstyle{point}
  {}     {}   
  {\normalfont}  
  {}       
  {\bfseries} 
  {}         
  {5pt plus 1pt minus 1pt} 
  {{\thmname{#1}(\thmnumber{#2})\thmnote{ {#3}.}}}          
\theoremstyle{point}
\newtheorem{point}[subsection]{}

\newcommand{\pa}[1]{\begin{point}#1\end{point}}              
\newcommand{\Pa}[2]{\begin{point}[#1]#2\end{point}}          

\newtheoremstyle{point*}
  {}     {}   
  {\normalfont}  
  {}       
  {\bfseries} 
  {}         
  {5pt plus 1pt minus 1pt} 
  {{\thmname{(#1)}\thmnote{ #3.}}}          
\theoremstyle{point*}
\newtheorem{point*}[subsection]{}
\newcommand{\upa}[2]{\begin{point*}[#1]#2\end{point*}}              

\numberwithin{equation}{section}

\newtheoremstyle{subpoint}
  {}     {}            
  {\normalfont}  
  {}                   
  {} 
  {}         
  {5pt plus 1pt minus 1pt} 
  {{\thmname{#1}(\thmnumber{#2})\thmnote{\textit{ #3.}}}}          
\theoremstyle{subpoint}
\newtheorem{subpoint}[subsubsection]{}

\usepackage{color}

\title[Triangulations of non-archimedean curves]{Triangulations of non-archimedean curves, semi-stable reduction, and ramification}

\subjclass[2010]{14D10 (primary) and 14G22, 14E22 (secondary)} 

\date{\today}

\author{Lorenzo Fantini}
\address{Goethe-Universit\"at Frankfurt, Institut f\"ur Mathematik, Frankfurt am Main, Germany}
\email{\href{mailto:fantini@math.uni-frankfurt.de}{fantini@math.uni-frankfurt.de}}
\urladdr{\url{https://lorenzofantini.eu/}}

\author{Daniele Turchetti}
\address{Dalhousie University, Department of Mathematics \& Statistics, Halifax, Nova Scotia, Canada}
\email{\href{mailto:daniele.turchetti@dal.ca}{daniele.turchetti@dal.ca}}
\urladdr{\url{https://www.mathstat.dal.ca/~dturchetti/}}

\begin{document}

\begin{abstract}
	Let $K$ be a complete discretely valued field with algebraically closed residue field and let $\mathfrak C$ be a smooth projective and geometrically connected algebraic $K$-curve of genus $g$.
	Assume that $g\geqslant 2$, so that there exists a minimal finite Galois extension $L$ of $K$ such that $\mathfrak C_L$ admits a semi-stable model.
	In this paper, we study the extension $L|K$ in terms of the \emph{minimal triangulation} of $C$, a distinguished finite subset of the Berkovich analytification $C$ of $\mathfrak C$.
	We prove that the least common multiple $d$ of the multiplicities of the points of the minimal triangulation always divides the degree $[L:K]$. 
	Moreover, in the special case when $d$ is prime to the residue characteristic of $K$, then we show that $d=[L:K]$, obtaining a new proof of a classical theorem of T. Saito.
	We then discuss curves with marked points, which allows us to prove analogous results in the case of elliptic curves, whose minimal triangulations we describe in full in the tame case.
	In the last section, we illustrate through several examples how our results explain the failure of the most natural extensions of Saito's theorem to the wildly ramified case.
\end{abstract}

\maketitle

\setcounter{tocdepth}{1}
\tableofcontents

\section{Introduction}

Let $R$ be a complete discrete valuation ring with algebraically closed residue field $k$ of characteristic $\mathrm{char}(k)=p\geqslant0$ and let $K$ be the fraction field of $R$.
In this paper, we are interested in studying the structure of smooth, projective, and geometrically connected curves over $K$, using techniques of non-archimedean analytic geometry.

A classical tool for understanding such a $K$-curve $\frakC$ comes from the celebrated semi-stable reduction theorem of Deligne and Mumford \cite{DeligneMumford1969}.
This fundamental result, proven in 1969, states that, after a suitable finite base field extension $L$ of $K$, the curve $\frakC$ acquires \emph{semi-stable reduction} over $L$, that is there exists a model $\calC$ of $\frakC_L$ over the valuation ring of $L$ whose special fiber is reduced and has at worst nodal singularities.
Such a model is called a \emph{semi-stable model} of $\frakC_L$.

Whenever the genus $g(\frakC)$ of $\frakC$ is at least 2, there exists a \emph{minimal} Galois extension $L$ of $K$ such that $\frakC$ acquires semi-stable reduction over $L$.
The following question, which has been extensively studied in the last half century, is therefore very natural.

\smallskip
\noindent{\bf Question:} What is the minimal Galois extension $L$ of $K$ such that $\frakC$ acquires semi-stable reduction over $L$?

\medskip

The situation is quite well understood whenever this minimal extension $L|K$ is \emph{tamely ramified}, that is its degree $[L:K]$ is prime to the residue characteristic $p$ of $K$.
Indeed, $\frakC$ acquires semi-stable reduction after a tamely ramified base field extension if and only if the special fiber $(\C_{\mathrm{min-snc}})_k$ of the minimal regular strict normal crossings (snc) model $\C_{\mathrm{min-snc}}$ of $\frakC$ over $R$ satisfies the following condition: the multiplicity of each \emph{principal component} of $(\C_{\mathrm{min-snc}})_k$ is prime to $p$.
Recall that an irreducible component of $(\C_{\mathrm{min-snc}})_k$ is said to be principal either if it has positive genus or if it intersects the rest of the special fiber of $\C_{\mathrm{min-snc}}$ in at least three points.

This criterion is a simple reformulation of a result proven by T. Saito (see \cite[Theorem 3.11]{Saito1987}).
Moreover, in this case Halle \cite[Theorem 7.5]{Halle2010} proved that the minimal extension $L$ of $K$ such that $\frakC$ acquires semi-stable reduction over $L$ is the unique tamely ramified extension of $K$ whose degree is the least common multiple of the multiplicities of the principal components of $(\C_{\mathrm{min-snc}})_k$.

However, in the general case, that is when it is required to perform a \emph{wildly ramified} base change of $K$, the question remains very poorly understood.
Indeed, in this case there appears to be no simple generalization of Saito's criterion that can hold true, as in general the degree of $L|K$ does not divide nor is divisible by the least common multiple of the multiplicities of the principal components of $(\C_{\mathrm{min-snc}})_k$.
Some conditions have conjecturally been proposed by Lorenzini \cite{Lorenzini2010}, but only partial results have been obtained, by Raynaud \cite{Raynaud1990} and Obus \cite{Obus2012b}, in the special case of a Galois covering of the projective line with cyclic $p$-Sylow group.
This suggests that the minimal regular snc model of $\frakC$ is not the correct object to look at in order to approach the question.

\medskip

In this paper, we propose to study this problem using tools of non-archimedean analytic geometry.
Denote by $C$ the \emph{analytification} of $\frakC$, that is the non-archimedean analytic curve associated with $\frakC$, in the sense of Berkovich theory \cite{Berkovich1990}.
The analytic curve $C$ is a Hausdorff, locally compact, and locally contractible topological space, whose global structure is deeply related to the combinatorics of the $R$-models of $\frakC$ and hence to the semi-stable reduction theorem.
Indeed, with each component of the special fiber of a normal model of $\frakC$ over $R$ we can associate a valuation which is a point of $C$.
Such points are called \emph{type 2} points of $C$, and with every \emph{vertex set} of $C$, that is a finite non-empty set of type 2 points of $C$, we can conversely associate a (possibly singular) normal $R$-model of $\frakC$.
Building upon results of Bosch--L\"utkebohmert \cite{BoschLuetkebohmert1985} (who worked within the framework of Tate's rigid geometry), this approach allows to deduce that $\frakC$ admits semi-stable reduction over $K$ if and only if $C$ can be decomposed as a disjoint union of a vertex set, finitely many \emph{open annuli}, and a family of \emph{open discs}.
This point of view can be exploited to deduce new proofs of Deligne and Mumford's theorem (see \cite{BoschLuetkebohmert1985}, \cite{Temkin2010}, and \cite{Ducros}).

From this viewpoint, the semi-stable reduction theorem is equivalent to the existence of a vertex set of $C$ such that the connected components of its complement become isomorphic to discs and annuli after a finite base change.
More precisely, we call \emph{virtual disc} any connected $K$-analytic space that becomes isomorphic to a union of open discs, and \emph{virtual annulus} any connected $K$-analytic space that becomes isomorphic to a union of open annuli, after passing to a finite separable extension $L$ of $K$.
Whenever such an extension $L$ of $K$, which is said to \emph{trivialize} the virtual disc or annulus, is tamely ramified, the $K$-analytic space we started with is well understood: in the case of virtual discs, this was studied by Ducros \cite{Ducros2013} and Schmidt \cite{Schmidt2015}, while the case of virtual annuli was the object of a previous work by the authors \cite{FantiniTurchetti2018} (see also \cite{Chapuis2017}).
In particular, tamely ramified virtual annuli are determined by whether they have one or two boundary points (or rather, more precisely, one or two \emph{ends}).

A \emph{triangulation} of $C$ is then defined as a vertex set whose complement in $C$ consists of virtual discs and virtual annuli.
This generalizes slightly a notion due to Ducros \cite{Ducros2008, Ducros}, who additionally required the annuli among the components of $C\setminus V$ to have two ends; we opted to modify Ducros's terminology and call this more restrictive version a \emph{strong triangulation} of $C$.
The starting point of the investigation of the present article is the fundamental fact that there exists a unique triangulation $V_{\mathrm{min-tr}}$ of $C$ that is minimal under inclusion.
As a consequence, it is simple to show that the minimal extension $L$ of $K$ such that $\frakC$ acquires semi-stable reduction over $L$ is the minimal extension that trivializes all the connected components of $C\setminus V_{\mathrm{min-tr}}$ (see Proposition~\ref{proposition_mintr_semistable_extension}), regardless of whether $L$ is tamely ramified over $K$ or not.
For this reason, the minimal triangulation $V_{\mathrm{min-tr}}$ of $C$, and its associated $R$-model, encodes the information needed to describe the extension $L|K$ even when the minimal regular snc model fails to do so.
Note that a unique minimal strong triangulation  $V_\mathrm{min-str}$ of $C$ also exists, but it is not as well-behaved under base change.

Observe that, if $x$ is a type 2 point of $C$, then the multiplicity of the exceptional component associated with $x$ in a suitable $R$-model of $C$ does not depend on the choice of the model, but only on the point $x$ (this multiplicity can also be seen as the index of the value group of $K$ in the value group of the completed residue field, in the sense of Berkovich's theory, at the point $x$).
We denote this multiplicity by $m(x)$.
Our first main result is the following.

\medskip 
\noindent{\bf Theorem (see Theorem~\ref{theorem_main}).}
{\it	Let $C$ be the analytification of a geometrically connected, smooth, and projective algebraic $K$-curve $\frakC$ of genus at least 2.
Let $V_{\mathrm{min-tr}}$ be the minimal triangulation of $C$ and let $L$ be the minimal Galois extension of $K$ such that $\frakC$ acquires semi-stable reduction over $L$.
Then
\[
\lcm\{m(x)\,|\,x\in V_{\mathrm{min-tr}}\}\,\Big|\,[L:K].
\]
}
\smallskip

We obtain this result by studying the behavior of multiplicities under base change (see Proposition~\ref{proposition_multiplicity_basechange}) and proving that the multiplicity of a type 2 point at the boundary of a disc or annulus must be one.

Moreover, we prove that the extension $L$ of $K$ in the theorem above is tamely ramified if and only if the residue characteristic of $K$ does not divide the least common multiple in the statement, in which case the latter actually coincides with the degree of $L|K$ (see \ref{theorem_main_effective}).
This allows to approach the classical Saito's criterion via a careful study of the vertex sets associated with the snc models of $\frakC$ (see Section~\ref{section_snc}), by relating the minimal strong triangulation $V_\mathrm{min-str}$ of $C$ to the minimal regular snc $R$-model $\C_\mathrm{min-snc}$ of $\frakC$ in the tame case.
Namely, we prove the following result.

\medskip 
\noindent{\bf Theorem (see Theorem~\ref{theorem_min-snc-triangulation}).}
{\it	Let $C$ be the analytification of a geometrically connected, smooth, projective algebraic $K$-curve $\frakC$ of genus at least 2.
Assume that $\frakC$ acquires semi-stable reduction after a tamely ramified extension of $K$.
Then the minimal strong triangulation of $C$ is the set of type 2 points associated with the principal components of the special fiber of the minimal regular snc model of $\frakC$ over $R$.
}
\smallskip\medskip

Moreover, we deduce from this result Halle's effective version of Saito's criterion (see Corollary~\ref{corollary_saito_effective}).

We prove the Theorem by first showing that the vertex set associated with the minimal regular snc model $\C_\mathrm{min-snc}$ of $\frakC$ is a triangulation, and then merging together virtual discs and annuli adjacent to non-principal type 2 points, which can be done thanks to Ducros' Fusion Lemmas (see Lemma~\ref{lemmas_fusion}).

In Section~\ref{section_elliptic_curves} we turn our attention to the study of elliptic curves.
While curves of genus less than two have no unique minimal triangulation, and no unique minimal extension of $K$ yielding semi-stable reduction exists, both problems can be resolved by endowing them with a finite set of marked rational points.
For example, in the case of an elliptic curve it is sufficient to mark its origin, and a unique minimal triangulation exists as soon as we require this point to be contained in a virtual disc rather than a virtual annulus.
By slightly modifying our approach to include marked points, both theorems above can be proven in the case of elliptic curves.
This gives us an interesting class of concrete examples.
Indeed, in \ref{point_elliptic_tame_additive} and \ref{point_elliptic_tame_multiplicative} we describe completely the minimal triangulation and the minimal strong triangulation of an elliptic curve in terms of its reduction type.
Conversely, we also show that the knowledge of the minimal triangulation of an elliptic curve, together with the isomorphism class of one component of its complement, is sufficient to retrieve the reduction type.

If $L$ is wildly ramified over $K$, then in general neither $V_{\mathrm{min-tr}}$ nor $V_{\mathrm{min-str}}$ contain, nor they are contained in, the set of type 2 points associated with the principal components of $(\C_\mathrm{min-snc})_k$.
This has several different explanations: examples discussing those pathologies are given in Section~\ref{section_wild}.
These include virtual discs whose minimal snc model has special fiber containing a principal component, and points that are associated with a rational curve in the special fiber of a $R$-model but become associated with a curve of positive genus after some base change;
 two phenomena that can only appear in the wildly ramified case.
The singularities of the connected components of $C\setminus V_{\mathrm{min-tr}}$ form a wide class of examples of the so-called \emph{wild quotient singularities}, and are therefore quite far from being understood; only a few very special cases have been recently investigated (see \cite{Lorenzini2013} and \cite{ObusWewers2020}). 
We believe that an approach via minimal triangulations can improve our understanding of this topic.
Minimal triangulations seem to be better suited than snc models to study the problem of semi-stable reduction in the wildly ramified case; we are convinced that in the future the approach proposed in this paper will contribute to shed some light on wild ramification.

\upa{Notation}{
	Let $K$ be a complete discretely valued field, let $R$ be its valuation ring and $k$ its residue field, and assume that $k$ is algebraically closed.
	Denote by $\pi$ a uniformizer of $R$ and endow $K$ with the $\pi$-adic absolute value such that $|\pi|=e^{-1}$.
	Denote by $p$ the characteristic exponent of $k$ (that is, $p=1$ if $\mathrm{char}(k)=0$ and $p=\mathrm{char}(k)$ otherwise).
	\\
	Let $\frakC$ be a geometrically connected, smooth, and projective algebraic curve of genus $g$ over $K$ and assume that $g\geqslant 2$.
	Finally, let $C$ be the non-archimedean $K$-analytic curve associated with $\frakC$, in the sense of Berkovich theory.
}

\upa{Acknowledgements}{
We are grateful to Antoine Ducros, Lars Halvard Halle, David Harbater, Qing Liu, Johannes Nicaise, Andrew Obus, and J\'er\^ome Poineau for several interesting discussions.
We are also extremely thankful to an anonymous referee for their thorough reading of the manuscript and numerous helpful comments.
During the preparation of this work, L.F. was partially supported by the project \emph{Lipschitz geometry of singularities (LISA)} of the \emph{Agence Nationale de la Recherche} (project ANR-17-CE40-0023) and by the \emph{Alexander von Humboldt Foundation}, while D.T. was partially supported by the \emph{European Research Council} (Starting Grant project ``TOSSIBERG'' no.637027), by the \emph{Fondation des Treilles}, and by the \emph{Atlantic Association for Research in the Mathematical Sciences}.
This article was written in part during a common stay at the \emph{Institut Henri Poincar\'e} in Paris, in the framework of the \emph{Research in Paris} program.
We thank the Institute for the warm hospitality and for the financial support.
}

\section{Non-archimedean curves}

In this section, we introduce non-archimedean curves and semi-affinoid spaces, and prove a simple result on the behavior of multiplicities of points under base change (Proposition~\ref{proposition_multiplicity_basechange}).

\pa{
	In this paper a \emph{curve} $X$ is a separated quasi-smooth strictly $K$-analytic space of pure dimension one, in the sense of Berkovich's theory of non-archimedean analytic geometry.
	In practice, all the analytic curves that we consider in this paper will be either the analytification $C=\frakC^\mathrm{an}$ of a smooth and projective algebraic curve $\frakC$ over $K$, or an open subspace of such a curve.
}

\pa{
	As our base field $K$ is not algebraically closed, throughout the paper we will often make use of the following general result \cite[Proposition 1.3.5.(i)]{Berkovich1990}.
	If $X$ is a $K$-analytic space and $L$ is a finite Galois extension of $K$ then the Galois group $\Gal(L|K)$ of $L|K$ acts continuously on the base change $X_L$ of $X$ to $L$ via automorphisms, and the base change morphism induces an isomorphism
	\(
	\bigslant{\textstyle X_L}{\Gal(L|K)} \stackrel{\sim}{\longrightarrow} X.
	\)
}

\Pa{Multiplicities}{
	A point $x$ of a curve $X$ is said to be a \emph{type 2} point if the residue field $\widetilde{\mathscr H(x)}$ of the complete residue field $\mathscr H(x)$ at $x$ is transcendental over $k$.
	If $x$ is a type 2 point of $X$, then by Abhyankar inequality the transcendence degree of $\widetilde{\mathscr H(x)}$ over $k$ is equal to 1 and the quotient of value groups $|\mathscr H(x)^\times| / |K^\times|$ is finite.
	The index ${\big[|\mathscr H(x)^\times|:|K^\times|\big]}$ is called the \emph{multiplicity} of $x$ and denoted by $m(x)$.
}

\begin{proposition}\label{proposition_multiplicity_basechange}
	Let $L$ be a finite, separable extension of $K$ of degree $d$, let $\tau\colon X_{L}\to X$ be the base change morphism, let $y$ be a type 2 point of $X_{L}$ and set $x=\tau(y)$.
	Then $m(x) \geqslant m(y)$ and $\frac{m(x)}{\gcd(m(x),d)}$ divides $m(y)$.
	If $L|K$ is tamely ramified or $p$ does not divide $m(x)$, then $m(y)=\frac{m(x)}{\gcd(m(x),d)}$.
\end{proposition}
	
\begin{proof}
	We have the following equalities of indices: 
	\begin{align*}
	\big[|\mathscr H(y)^\times|:|K^\times|\big] 
	& = \big[|\mathscr H(y)^\times|:|L^\times|\big] \cdot \big[|L^\times|:|K^\times|\big] = m(y) \cdot d\\
	& = \big[|\mathscr H(y)^\times|:|\mathscr H(x)^\times|\big] \cdot \big[|\mathscr H(x)^\times|:|K^\times|\big]\\
	& = \big[|\mathscr H(y)^\times|:|\mathscr H(x)^\times|\big] \cdot m(x).
	\end{align*}
	It follows that, since $\big[|\mathscr H(y)^\times|:|\mathscr H(x)^\times|\big] \leqslant \big[\mathscr H(y):\mathscr H(x)\big] \leqslant [L : K] = d$, we have $m(x)\geqslant m(y)$.
	Moreover, the group $|\mathscr H(y)^\times|$ contains the absolute values of the uniformizers of $L$ and $\mathscr H(x)$, which respectively have order $d$ and $m(x)$ in the quotient $|\mathscr H(x)^\times| / |K^\times|$.
	As a consequence, $\big[|\mathscr H(y)^\times|:|K^\times|\big]$ is a multiple of $\lcm\big(d, m(x)\big)=\frac{d\cdot m(x)}{\gcd(m(x),d)}$, which is to say that $\frac{m(x)}{\gcd(m(x),d)}$ divides $m(y)$, and the first part of the proposition is proved.
	Moreover, observe that the equality holds if and only if the value group $|\mathscr H(y)^\times|$ is generated by $|L^\times|$ and $|\mathscr H(x)^\times|$ over $|K^\times|$.
	Let us show that this is the case under the assumption that $L|K$ is tamely ramified.
	In this case, we have $L=\frac{K[Z]}{(Z^d-\pi)}$ and, since the residue field $k$ of $K$ is algebraically closed, $L$ is a Galois extension of $K$.
	We now study two situations that complement each other.
	First, we suppose that $\mathscr H(x)$ contains $L$.
	Since $L|K$ is Galois, we have that ${\mathscr H(x) \widehat{\otimes}_K L \cong \bigoplus_{i=1}^d \mathscr H(x)}$, so that $\mathscr H(y)=\mathscr H(x)$.
	It follows that $m(y)=\frac{m(x)}{d}=\frac{m(x)}{\gcd(m(x),d)}$.
	At the opposite side of the picture, if $\mathscr H(x) \cap L = K$ then $\mathscr H(y) \cong \frac{\mathscr H(x)[Z]}{(Z^d-\pi)}$.
	Then, the needed result is a consequence of the fact that in $\mathscr H(y)$ we have $\left|\sum_{i=0}^{d-1} \lambda_i \varpi^i \right| = \max\{|\lambda_i||\varpi^i|\}$ for every $d$-uple $(\lambda_i)$ in $\mathscr H(x)^d$. 
	This can be proved using Temkin's theory of graded reduction by observing that the set $\{1, \varpi, \dots, \varpi^{d-1}\}$ is a basis for the extension of graded reductions $\widetilde{\mathscr H(y)}|\widetilde{\mathscr H(x)}$, a fact proven in \cite[2.21]{Ducros2013} (see the totality of section 2 of \emph{loc.\ cit}, and in particular Lemma 2.3, for the relevant definitions in graded algebra).
	The case of a general tamely ramified extension can be deduced by considering first the base change of $X$ from $K$ to $\mathscr H(x) \cap L$, which being tamely ramified is itself Galois over K because the latter contains all prime-to-$p$ roots of unity, and then the base change from $\mathscr H(x) \cap L$ to $L$.
	Finally, if $L|K$ is arbitrary and $p$ does not divide $m(x)$, then we can prove the claim at the end of the statement by splitting the extension $L|K$ as a purely wildly ramified extension of the maximally tame extension of $K$.
	Indeed, in the first case we have $\gcd(m(x),d)=1$ and thus $m(x)$ divides $m(y)$ by the first part of the proposition, so that the inequality $m(y)\leqslant m(x)$ is in fact an equality, while the tame case has been previously treated. 
\end{proof}


\Pa{Genus of a point}{
	With each type 2 point $x$ of a $K$-curve $X$ we can associate a numerical invariant, its \emph{genus}, as follows.
	The set $y_1,\ldots,y_n$ of points above $x$ in the base change $X_\Kalg$ of $X$ to $\Kalg$ is finite by \cite[4.5.1]{Ducros} (or \cite[Lemma 1.1.5]{Conrad1999}).
	The group $\mathrm{Aut}\big(\widehat{K^\mathrm{alg}}|K\big)$ of continuous automorphisms of $\widehat{K^\mathrm{alg}}$ fixing $K$ acts via isomorphisms on $X_\Kalg$, inducing a transitive permutation on the set $\{y_i\}$ and hence isomorphisms $\mathscr H(y_i) \cong \mathscr H(y_j)$ for every $i$ and $j$.
	We define the \emph{genus} $g(x)$ of $x$ to be the genus of the unique smooth projective $\tilde{k}$-curve whose function field is $\widetilde{\mathscr H(y_i)}$.
}

\Pa{Semi-affinoid spaces}{
Let us quickly recall some basic facts about an important class of $K$-analytic spaces, that of semi-affinoid spaces.
We will adopt an ad hoc point of view in order to keep things as simple as possible; a more thorough treatment and further references can be found in \cite[Section 2]{FantiniTurchetti2018}.

Let 
\(
D_{m,n}
 = 
D^m_K\times_K (D^-_K)^{n} 
\)
be the product of the $m$-dimensional $K$-analytic closed unit disc with the $n$-dimensional $K$-analytic open unit disc, that is the subset of the $(n+m)$-dimensional $K$-analytic affine line $\mathbb A^{m+n,\mathrm{an}}_K$ cut out by the inequalities $|Y_i(x)|\leqslant1$ for all $i=1,\ldots,m$ and $|Y_i(x)|<1$ for all $i=m+1,\ldots,m+n$, where the $Y_i$'s are coordinates on $\mathbb A^{m+n,\mathrm{an}}_K$.
Denote by 
\[
\calO^\circ(D_{m,n})=R\{Y_1,\ldots,Y_m\}\lbrack\lbrack Y_{m+1},\ldots,Y_{m+n} \rbrack\rbrack
\]
the algebra of the analytic functions on $D_{m,n}$ that are bounded by $1$ in absolute value.
A \emph{semi-affinoid} $K$-analytic space $U$ is any space of the form
\[
U=U_{m,n,(f_1,\ldots,f_r)}=V(f_1,\ldots,f_r)\subset D_{m,n},
\]
for some integers $m,n\geqslant0$ and analytic functions $f_1,\ldots,f_r\in \calO^\circ(D_{m,n})$, where $V(f_1,\ldots,f_r)$ denotes the zero locus of the functions $f_i$ in $D_{m,n}$.
}

\pa{
Let $U=U_{m,n,(f_1,\ldots,f_r)}$ be a semi-affinoid $K$-analytic space and consider the $R$-algebra 
\[
A=A_{m,n,(f_1,\ldots,f_r)}=\frac{R\{Y_1,\ldots,Y_m\}\lbrack\lbrack Y_{m+1},\ldots,Y_{m+n} \rbrack\rbrack}{(f_1,\ldots,f_r)}.
\]
Any such algebra is usually called a \emph{special $R$-algebra}.
Assume that the algebra $A$ is reduced, so that $U$ is a reduced $K$-analytic space. Then, the integral closure of $A$ in the $K$-algebra $A\otimes_RK$ coincides with the set of analytic functions on $U$ that are bounded by 1 in absolute value (and so in particular it does not depend on $m,n$, and $(f_1,\ldots,f_r)$, but only on the isomorphism class of $U$).
It has a natural $R$-algebra structure and we denote it by $\calO^\circ(U)$.
The $R$-algebra $\calO^\circ(U)$ determines the semi-affinoid space $U$ completely: $\calO^\circ(U)\cong A_{m',n',(g_1,\ldots,g_{r'})}$ is itself a special $R$ algebra and $U\cong U_{m',n',(g_1,\ldots,g_{r'})}$ (see \cite[Lemma 2.2]{FantiniTurchetti2018}).
}

\pa{
The good point of view for studying semi-affinoid spaces is that of formal schemes.
Indeed, a special $R$-algebra $A$ as above is a noetherian adic topological ring, with ideal of definition $J$ generated by the uniformizer of $R$ and by the coordinate functions $Y_{m+1},\ldots,Y_{n+m}$, so that we can consider its formal spectrum $\Spf(A)=\varinjlim_{n}\Spec(A/J^n)$.
The semi-affinoid space $U$ is then the $K$-analytic space associated with $\Spf(A)$, as in \cite[Section 7]{deJong1995} and \cite[Section 1]{Berkovich1996}.
We call the formal scheme $\Spf\big(\calO^\circ(U)\big)$ the \emph{canonical model} of $U$.
}

\Pa{Field of constants}{
\label{def_fieldofconstants}
	Given a connected $K$-analytic curve $X$, we define its \emph{field of constants} $\mathfrak s(X)$ as
	\[
	\mathfrak s(X) = \big\{ f \in \mathcal O_X(X) \,\big|\, P(f)=0 \text{ for some }  P\in K[T]\text{ separable} \big\}.
	\]
	It is a finite separable field extension of $K$ which is contained in the sub-algebra of $\mathcal O_X(X)$ and consists of those functions that are constant on $X$.
	The degree $[\mathfrak s(X): K]$ coincides with the number of connected components of $X \times_K \fraks(X)$, and hence of the base change of $X$ to the algebraic closure of $K$.
	In particular, if $X$ has a $L$-rational point for some finite separable extension $L$ of $K$ then $\fraks(X)$ is contained in $L$.
}

\pa{
	\label{def_Us} 
	\label{definition_space_over_field_of_constants}
	Since the non-archimedean field $\fraks(X)$ is contained in the $K$-algebra $\mathcal O_X(X)$, the $K$-analytic space $X$ has also a natural structure of an $\mathfrak s(X)$-analytic space.
	We denote by $X_\fraks$ this $\fraks(X)$-analytic space.
	Since $X_\fraks$ is geometrically connected, this procedure often gives a convenient way to avoid having to deal with curves that are not geometrically connected.
	Indeed, if $L$ is a Galois extension of $K$ containing $\fraks(X)$, then the base change $X \times_K L$ of $X$ to $L$ is isomorphic to $[\fraks(X):K]$-many copies of $X_\fraks\times_{\fraks(X)}L$.
	If moreover $X$ is a semi-affinoid space, then $\fraks(X)^\circ$ is the integral closure of $R$ in $\mathcal O^\circ(X)$, and $X_\fraks$ is naturally a semi-affinoid space as well simply by seeing $\mathcal O^\circ(X)$ as a special $\fraks(X)^\circ$-algebra.
}

\pa{\label{point_splitting_function}
	With each $x \in X$ we associate the largest subfield $\mathfrak s(x)$ of $\mathcal H(x)$ that is separable over $K$.
	We consider the map 
	\begin{align*}
	s\colon & X \longrightarrow \N \\
	&x \longmapsto [\mathfrak s(x):K],
	\end{align*} 
	which we call the \emph{splitting function} of $X$.
	For every semi-affinoid subspace $U$ of $X$ containing $x$ and every constant function $f\in \mathfrak s(U)$, the evaluation at $x$ produces an element $f(x)\in \mathfrak s(x)$.
	This induces an injection $\mathfrak s(U) \hookrightarrow \mathfrak s(x)$.
	Moreover, the local ring $\mathcal O_{X,x}$ being Henselian, it contains $\fraks(x)$, and so there exists a semi-affinoid neighborhood $U$ of $x$ contained in $X$ such that $\fraks(x) = \fraks(U)$.
	As a consequence, the splitting function $s$ is lower semi-continuous on $X$.
	Observe that from the discussion above it follows that $s(x)$ is the cardinality of the preimage of $x$ in the base change $X_\Kalg$ of $X$ to $\Kalg$.
}

\pa{
	Given a connected open $K$-analytic curve $X$, the number of connected components of $X \setminus \{Y\}$, where $Y$ is a compact subset of $X$, does not depend on $Y$ if $Y$ is big enough.
	This allow us to define the set of \emph{ends} of $X$ as 
	\[
	\Ends(X)= \varprojlim_{Y} \pi_0(X\setminus Y).
	\]
	If $\varepsilon\in \Ends(X)$ is an end of $X$, then it can be represented by a family $(U_n)_{n\in \N}$ of connected open subsets $U_n$ of $X$ such that $U_{n+1}\subset U_{n}$ for every $n$ and $\cap_n \Ends(U_n)=\{\varepsilon\}$.
}

\pa{
	Let $\varepsilon\in \Ends(X)$ be an end of $X$.
	The field of constants $\fraks(U_n)$ of $U_n$ does not depend on the choice of a representative $(U_n)_{n}$ of $\varepsilon$ as above nor on $U_n$ if $n$ is big enough.
	We denote this field by $\fraks(\varepsilon)$ and call it the \emph{field of constants} of $\varepsilon$.
	Its degree $[\fraks(\varepsilon):K]$ over $K$ counts the number of ends of $X_\Kalg$ above $\varepsilon$.
}

\section{Virtual discs and virtual annuli}

In this section we introduce two very important families of semi-affinoid $K$-analytic spaces, namely virtual discs and virtual annuli, give several examples of those and discuss some of their fundamental properties.

\pa{
We call \emph{open disc} (or simply \emph{disc}) over $K$ any $K$-analytic curve isomorphic to
\[
D_K=\big\{x\in\mathbb A^{1,\mathrm{an}}_K\;\big|\;|T(x)|<1\big\}\,,
\]
where $T$ denotes the standard coordinate function on $\A^{1,\mathrm{an}}_K$.
It is the semi-affinoid $K$-analytic space associated with the special $R$-algebra $R[[T]]$.
We call \emph{open annulus} (or simply \emph{annulus}) over $K$ any $K$-analytic curve isomorphic to
\[
A_{n,K}=\big\{x\in\mathbb A^{1,\mathrm{an}}_K\;\big|\;|\pi^n|<|T(x)|<1\big\}
\]
for some $n\in\N$.
It is the semi-affinoid $K$-analytic space associated with the special $R$-algebra $R[[S,T]]/(ST-\pi^n)$.
}

\begin{figure}[h]
	\begin{tikzpicture}[scale=.8]	
	\draw[thin] (0,0)--(0,4.25);
	\draw[thin,dashed] (0,4.25)--(0,5);
	\draw[thin] (0,.75)--(.5,.2);
	\draw[thin] (0,1.5)--(.8,.7);
	\draw[thin] (0,2.75)--(1.25,1.75);
	\draw[thin] (0,3.375)--(1.75,2.75);
	\draw[thin] (0,4.25)--(3,4);
	\draw[thin] (1.5,4.125)--(2.5,4.75);

	\draw[fill] (0,0)circle(1pt);
	\draw[fill] (.5,.2)circle(1pt);
	\draw[fill] (.8,.7)circle(1pt);
	\draw[fill] (1.25,1.75)circle(1pt);
	\draw[fill] (1.75,2.75)circle(1pt);
	\draw[fill] (3,4)circle(1pt);
	\draw[fill] (2.5,4.75)circle(1pt);

	\begin{small}
	\node(a)at(-.6,4.85){$\aan$};
	\end{small}

	\begin{scriptsize}
	\node(a)at(-.2,.05){$0$};
	\node(a)at(.7,.1){$\pi^3$};
	\node(a)at(1,.6){$\pi^2$};
	\node(a)at(1.45,1.65){$\pi$};
	\node(a)at(1.7,2.45){$[\sqrt{\pi}]$};
	\node(a)at(3.15,3.8){$1$};
	\node(a)at(3.05,4.75){$1+\pi^3$};
	\end{scriptsize}

\draw[fill=darkblue, opacity=.2] plot [smooth cycle, tension=1.5] coordinates {(1.45,4) (3.2,4.25) (3.1,3.5)};
\begin{footnotesize}
\node(a)at(5.4,4.15){{\color{darkblue}$\big\{|T(x)- 1|<|\pi^3|\big\}$}};
\node(a)at(4.2,3.6){{\color{darkblue}disc}};
\end{footnotesize}

\draw[fill=red, opacity=.2] plot [smooth cycle, tension=1.5] coordinates {(-.5,4.2) (2.5,2.2) (0,1.5)};
\begin{footnotesize}
\node(a)at(5.15,2.25){{\color{red}$A_{2,K}=\big\{|\pi^2|<|T(x)|<1\big\}$}};
\node(a)at(3.4,1.75){{\color{red}annulus}};
\end{footnotesize}

	\end{tikzpicture}
	\caption{\label{figure:disc_annulus}
	A disc, in blue, and the annulus $A_{2,K}$, in red, seen as subspaces of $\aan$.}
\end{figure}
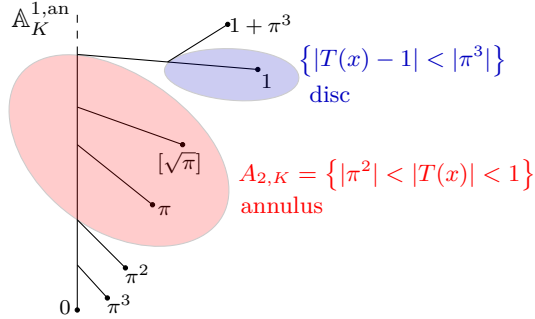

\Pa{Virtual discs and virtual annuli}{
Let $X$ be a connected semi-affinoid curve over $K$.
We say that $X$ is a \emph{virtual open disc} (or simply a \emph{virtual disc}) if there exists a finite separable extension $L$ of $K$ such that $X_L=X \times_K L$ is a disjoint union of open discs over $L$.
We say that $X$ is a \emph{virtual open annulus} (or simply a \emph{virtual annulus}) if there exists a finite separable extension $L$ of $K$ such that $X_L$ is a disjoint union of open annuli over $L$.
In both cases, we say that $X$ \emph{is trivialized} by $L$.
}

\pa{
Let $X$ be a virtual disc (respectively a virtual annulus) over $K$, trivialized by $L$. 
Since discs and annuli are geometrically connected, $\fraks(X)$ embeds in $L$.
The curve $X_\fraks$ being geometrically connected, its base change $X_\fraks \times_{\fraks(X)} L$ to $L$ is then a disc (respectively an annulus) over $L$.
Conversely, any $K$-analytic curve $X$ such that $X_\fraks$ becomes a disc (respectively  an annulus) after a finite separable base field extension $L'$ of $\fraks(X)$ is a virtual disc (respectively a virtual annulus) since $X \times_K L'$ consists of a finite number of copies of $X_\fraks\times_{\fraks(X)}L'$.
}

\pa{
	It follows from the definition that any disc has one end and any annulus has two ends.
	Let $X$ be a virtual disc or a virtual annulus trivialized by a finite Galois extension $L$ of $K$.
	Since $X$ is homeomorphic to $X_\fraks$, by replacing the former with the latter we can assume that $X$ is geometrically connected.
	Then $\Gal(L|K)$ acts on $\Ends(X_L)$ and we have $\Ends(X) \cong {\Ends(X_L)}/{\Gal(L|K)}$.
	Therefore, if $X$ is a virtual disc then $\Ends(X)$ consists of a single point.
	On the other hand, if $X$ is a virtual annulus then $\Ends(X)$ consists of two points if the two elements of $\Ends(X_L)$ are fixed by every element of $\Gal(L|K)$, and of a single point otherwise.
}

\Pa{Terminology}{
	We caution the reader on our terminology here, since it differs from the one of \cite{Ducros}.
	Indeed, in \emph{loc. cit.} virtual annuli are required to have two ends, which is not the case here.
	In turn, this will lead us to a definition of triangulation different than the one of \emph{loc. cit.} (see \ref{point:def_triangulation_terminology}).
}

\Pa{Skeletons}{ \label{skeletons_virtual_annuli}
Discs and annuli are connected, contractible special quasi-polyhedrons, in the sense of \cite[Section 4.1]{Berkovich1990}.
If $X$ is an annulus, then we call \emph{skeleton} of $X$ the real open interval connecting the two ends of $X$, and we denote it by $\Sigma(X)$.
It follows readily from the definition of annulus that $\Sigma(X)$ coincides with the \emph{analytic skeleton} of $X$, which is the set of points of $X$ that have no neighborhood in $X$ isomorphic to a virtual disc, and with the \emph{topological skeleton} of $X$, which is the set of points of $X$ that have no neighborhood in $X$ homeomorphic to a real tree with one end.
Moreover, $X$ admits a strong deformation retraction onto $\Sigma(X)$ which is topologically a proper map (this follows for example from \cite[Proposition 4.1.6]{Berkovich1990}).
Now, let $X$ be a virtual annulus trivialized by an extension $L$ of $K$, denote by $\tau\colon X_L\to X$ the base change morphism, and set $\Sigma(X)=\tau\big(\Sigma(X_L)\big)$, where $\Sigma(X_L)$ is the union of the skeletons of the components of $X_L$.
The action of the Galois group $\Gal(L|K)$ of $L$ over $K$ fixes $\Sigma(X_L)$ and identifies $\Sigma(X)$ with the quotient ${\Sigma(X_L)}/{\Gal(L|K)}$.
To understand this quotient, by replacing the virtual annulus $X$ with $X_\s$ we can assume without loss of generality that $X$ is geometrically connected, so that $\Sigma(X_L)$ is a single line segment.
Then the action of an element $\sigma$ of $\Gal(L|K)$ on $\Sigma(X_L)$, being a continuous automorphism of finite order, is either the identity or a non-trivial involution permuting the ends of the line segment.
In both cases $\Sigma(X)$, which we call again \emph{skeleton} of $X$, is itself a line segment (either open or semi-open) which coincides with the analytic skeleton of $X$.
In the first case $\Sigma(X)$ is also the topological skeleton of $X$, while in the second case the analytic structure of $X$ is necessary to determine it.
It follows again by \cite[Proposition 4.1.6]{Berkovich1990} that $X$ admits a topologically proper strong deformation retraction onto $\Sigma(X)$. 
}

\pa{
Let $X$ be a virtual annulus with two ends, so that the skeleton $\Sigma(X)$ of $X$ is the open line segment connecting the two ends, and let $x$ be a point of $\Sigma(X)$.
Then $\fraks(x)=\fraks(X)$.
To see this, we can assume without loss of generality that $X$ is geometrically connected, since $\fraks(x)$ is also the constant field of $x$ seen as a point of $X_\s$.
Then our claim follows from the fact that $[\s(x):K]$ is the number of preimages of $x$ in the skeleton $\Sigma(X_L)$ of the annulus $X_L$, where $L$ is a finite extension of $K$ trivializing $X$, and we have already seen that in this case the Galois group of $L|K$ fixes $\Sigma(X_L)$ point-wise.
}

\pa{\label{point_break_one-ended_annuli}
	Let $X$ be a virtual annulus with one end, so that the skeleton $\Sigma(X)$ of $X$ is a semi-open line segment connecting the end of $X$ to a point $x$ of $\Sigma(X)$.
	We call the point $x$ the \emph{bending point} of $X$.
	Then it follows from the discussion of \ref{skeletons_virtual_annuli} that the connected component of $X\setminus\{x\}$ that intersects $\Sigma(X)$ nontrivially is a virtual annulus with two ends, while all other connected components are virtual discs.
	If we assume again that $X$ is geometrically connected, then every point $y$ of the skeleton $\Sigma(X)$ of $X$ different from the bending point $x$ has two distinct inverse images in the base change $X_\Kalg$ of $X$ to $\Kalg$ (that is, $s(y)=2$), while $x$ has one (that is, $s(x)=1$), as is depicted in the example of Figure~\ref{figure:virtual_disc_annulus}.
}

\begin{examples}\label{examples:virtual}
	We begin by giving some examples of geometrically connected virtual discs and annuli, chosen among the ones that will play a relevant role in the rest of the paper.
	\begin{enumerate}
		\item \label{example_fractional_disc}
		Consider the subspace
		\[
		X=\big\{x\in\aan \;\big|\; |T(x)|<|\pi|^{\sfrac{1}{d}}\big\}
		\]
		of the analytic affine line $\A^{1,\mathrm{an}}_K$.
		It is the semi-affinoid space associated with the special $R$-algebra $R[[S,T]]/(\pi S- T^d)$.
		This is what we call a \emph{fractional disc} over $K$.
		Observe that $X$ has a $K$-rational point, hence in particular $\fraks(X)\cong K$.
		Its boundary in $\A^{1,\mathrm{an}}_K$ consists of a single point of multiplicity $d$.
		There exists a finite separable field extension $L$ of $K$ such that $L$ contains an element $\varpi$ of absolute value $|\varpi|=|\pi|^{1/d}$ (for example, if $d$ is prime with the residue characteristic of $K$ or if $K$ is perfect we can take for $\varpi$ a $d$-th root of $\pi$), then $X_L$ is the disc 		
		\(
		\big\{x\in\A^{1,\mathrm{an}}_L \;\big|\; |T(x)|<|\varpi|\big\}
		\)
		over $L$, and so $X$ is a virtual disc.
		

		\item \label{example_Ducros2013}
		Any geometrically connected virtual disc over $K$ that is trivialized by a tamely ramified extension of $K$ is of the form described in the previous example.
		This follows from \cite[Th\'eor\`eme 3.5]{Ducros2013}.

		\item \label{example_FantiniTurchetti2018}
		Any geometrically connected virtual annulus $X$ over $K$ with two ends that is trivialized by a tamely ramified extension of $K$ is of the form 
		\[
		X=\big\{x\in\aan \;\big|\; |\pi|^\beta < |T(x)| < |\pi|^\alpha\big\}		
		\]
		for some $\alpha,\beta\in\Q$, as proved in \cite[Theorem 8.1]{FantiniTurchetti2018}.
		This is what we call a \emph{fractional annulus} over $K$.

		\item \label{example_disco_stronzo} 
		Suppose that $K$ is of mixed characteristic $(0,p)$ and consider the subspace $X$ of $\aan$ defined as
		\[
		X=\big\{x\in\aan \;\big|\; |(T^p-\pi)(x)|<|\pi|\big\}. 
		\]
		Then $X$ is the semi-affinoid $K$-analytic curve having special $R$-algebra $R[[U,V]]/(U^p-\pi-\pi V)$.
		Its boundary in $\aan$ consists of the point $x$ on the path from 0 to the Gauss point satisfying $|T(x)|=|\pi|^{\sfrac{1}{p}}$, which has multiplicity $p$.
		After choosing a $p$-th root $\varpi$ of $\pi$ and performing a base change to the wild extension $L=K(\varpi)$ of $K$, the space $X_L$ is isomorphic to the disc
		\[
		X_L=\big\{x\in\A^{1,\mathrm{an}}_{L} \;\big|\; |(T-\varpi)(x)|<|\varpi|\big\}
		\]
		of center $\varpi$ and radius $|\varpi|=|\pi|^{\sfrac{1}{p}}$, since the distance between two distinct $p$-th roots of $\pi$ is smaller than $|\varpi|$.
		However, despite being geometrically connected, the virtual disc $X$ is not itself a disc with a rational radius as in $(i)$, since it has no rational point over $K$.
		More generally, other examples of virtual discs trivialized by a wildly ramified extension of $K$ can be produced by picking an Eisenstein polynomial $P(S)\in R[S]$ of degree $p^n$ and considering the special $R$-algebra $R[[S,T]]/\big(P(S)-\pi T\big)$.
		The associated semi-affinoid space is a virtual disc trivialized by the extension $K[S]/\big(P(S)\big)$.
		Examples of geometrically connected virtual discs over $K$ that are not discs and yet have a $K$-rational point also exist, see for example \cite[3.8.2]{Ducros2013}.
	\end{enumerate}
\end{examples}

\pa{
	Let $X$ be a connected semi-affinoid curve over $K$.
	If $X_\fraks$ is isomorphic to a fractional disc
	\[
	\{x\in \mathbb{A}^{1,\mathrm{an}}_{\fraks(X)} \,\big|\, |T(x)| < \rho\big\}
	\]
	for some $\rho \in \sqrt{|K^\times|}$, then $X$ is a virtual disc which we call \emph{generalized fractional disc}.
	If $X_\fraks$ is isomorphic to a fractional annulus
	\[
	\big\{x\in \mathbb{A}^{1,\mathrm{an}}_{\fraks(X)} \,\big|\, \rho_1<|T(x)| < \rho_2\big\}
	\]
	for some $\rho_1, \rho_2 \in \sqrt{|K^\times|}$, then $X$ is a virtual annulus which we call \emph{generalized fractional annulus}.
}

\begin{examples}
Here are some examples of generalized fractional discs and annuli.
	\begin{enumerate}
		\item \label{example:virtual_several_fractional_discs}
		Let $m$ be a positive integer prime with the residue characteristic of $K$.
		Similarly as in Example~\ref{example_disco_stronzo}, the subspace $X$ of $\aan$ defined as
		\[
		X=\big\{x\in\aan \;\big|\; |(T^m-\pi)(x)|<|\pi|\big\}. 
		\]
		is a semi-affinoid $K$-analytic curve associated with the special algebra $R[[U,V]]/(U^m-\pi-\pi V)$, whose boundary in $\aan$ consists of a point of multiplicity $m$.
		However, unlike in the case of Example~\ref{example_disco_stronzo}, if $\varpi$ and $\varpi'$ are two distinct $m$-th roots of $\pi$ in $\Kalg$, then $|\varpi-\varpi'|=|\varpi|$, and hence the base change $X_L$ of $X$ to the tamely ramified extension $L=K(\varpi)$ of $K$ is
		\[
		X_L=\bigcup_{\zeta^m=1}\big\{x\in\A^{1,\mathrm{an}}_{L} \;\big|\; |(T-\zeta\varpi)(x)|<|\varpi|\big\},
		\]
		which is a disjoint union of $m$ discs over $L$.
		In particular, we deduce that the field of constants $\s(X)$ of $X$ is $L$ and that $X$ is a generalized fractional disc.
		This can also be seen by observing that there exists an isomorphism of special $R$-algebras 
		\[R[[U,V]]/(U^m-\pi-\pi V) \to R[[S,T]]/(S^m- \pi),\] 
		defined by sending $U$ to $S(1+T)^{1/m}$ and $V$ to $T$.
		These two algebras are not isomorphic in the situation of Example~\ref{example_disco_stronzo}, as can be expected since the one on the right contains the constant function $S$.
		A picture of such a virtual disc for $m=2$ is given on the left of Figure~\ref{figure:virtual_disc_annulus}
		
		\item 
		Let $\alpha$ and $\beta$ be two rational numbers with $\beta>\alpha\geqslant 1$ and let $d$ be a positive integer prime with the residue characteristic of $K$.
		Then the subspace of $\aan$ of the form
		\[
		X=\big\{x\in\aan \;\big|\; |\pi|^\beta < |(T^d-\pi)(x)| < |\pi|^\alpha\big\}		
		\]
		is a generalized fractional annulus.
		Its field of constants is $\s(X)=K(\pi^{\sfrac{1}{d}})$ and its ends have multiplicities $da$ and $db$, where $a$ and $b$ are the smallest positive integers such that $a\alpha$ and $b\beta$ are integers.
		Indeed, over $\s(X)$ we have 
		\[
		X_\s \cong \big\{x\in\A^{1,\mathrm{an}}_{\s(X)} \;\big|\; |\pi|^{\sfrac{\beta}{d}} < |T(x)| < |\pi|^{\sfrac{\alpha}{d}}\big\}.
		\]

		\item \label{examples_tame_forms_generalized}
		It follows easily from Examples~\ref{example_Ducros2013} and \ref{example_FantiniTurchetti2018} that any virtual disc (respectively virtual annulus with two ends) over $K$ trivialized by a tamely ramified extension of $K$ is a generalized fractional disc (respectively a generalized fractional annulus).
	\end{enumerate}
\end{examples}

\begin{example}\label{example_virtual_annulus_oneend}
Let $n>1$ be an integer, assume that the residue characteristic of $K$ is different from 2, and consider the subspace of $\pan$ defined by
\[
X = \{x\in\pan \;\big|\; |(T^2-\pi)(x)|>|\pi|^n\}
\]
and depicted at the bottom-right of Figure~\ref{figure:virtual_disc_annulus}.
The complement of $X$ in $\pan$ is a closed version of a generalized fractional disc analogous to the one of Example \ref{example:virtual_several_fractional_discs} (with $m=2$ and the radius raised to the power $n$).
It follows that the base change $X_L$ of $X$ to the quadratic extension $L=K(\sqrt{\pi})$ of $K$ is the complement in $\pan$ of two closed discs, which is an annulus.
Since $X$ has only one end, as its boundary in $\pan$ consists of a single point, this shows that $X$ is a virtual annulus with one end, trivialized by $L$.		
In fact, we have shown in \cite[Theorem 8.3]{FantiniTurchetti2018} that, whenever the residue characteristic of $K$ is different from two, any virtual annulus with one end that is trivialized by a quadratic extension of $K$ is isomorphic to the $K$-analytic space $X$ (for some choice of $n>1$).
The explicit equations appearing in \emph{loc. cit.} also show that $X$ is the semi-affinoid $K$-analytic space associated with the special $R$-algebra
		\[
		R[[S,T,U]]/(S^2-U+\pi^n, \pi U-T^2).
		\]
\end{example}

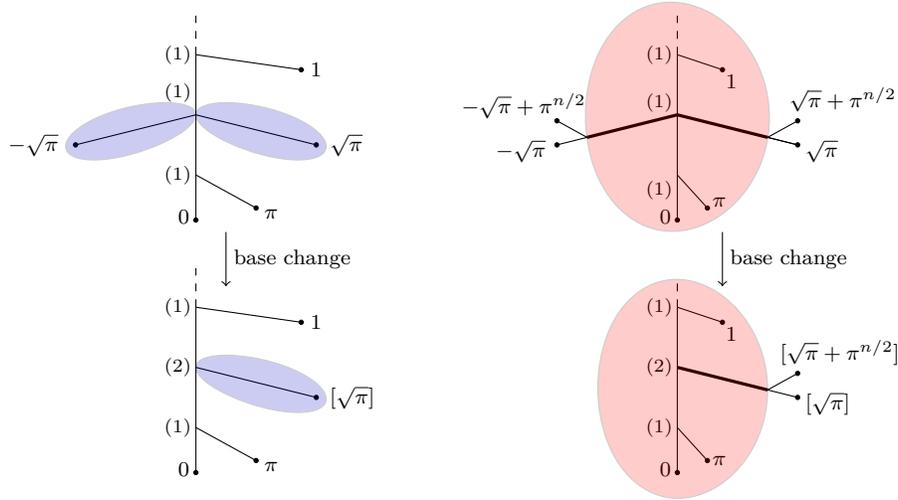
\begin{figure}[h]
	\begin{tikzpicture}[scale=.8]
	\draw[thin] (0,0)--(0,2.75);
	\draw[thin,dashed] (0,2.75)--(0,3.4);
	\draw[thin] (0,.75)--(1,.2);
	\draw[thin] (0,1.75)--(2,1.25);
	\draw[thin] (0,2.75)--(1.75,2.5);

	\draw[fill] (0,0)circle(1pt);
	\draw[fill] (1,.2)circle(1pt);
	\draw[fill] (2,1.25)circle(1pt);
	\draw[fill] (1.75,2.5)circle(1pt);

	\begin{scriptsize}
	\node(a)at(-.2,.05){$0$};
	\node(a)at(1.25,.1){$\pi$};
	\node(a)at(2.6,1.25){$[\sqrt{\pi}]$};
	\node(a)at(2,2.5){$1$};
	\end{scriptsize}
	
	\begin{tiny}
	\node(a)at(-.3,.75){$(1)$};
	\node(a)at(-.3,1.75){$(2)$};
	\node(a)at(-.3,2.75){$(1)$};
	\end{tiny}

	\draw[fill=darkblue, opacity=.2] plot [smooth cycle, tension=1.5] coordinates {(0,1.75) (1.65,1.7) (1.65,1)};

	\begin{scope}[yshift=4.2cm]
	
	\draw[thin,->] (.5,-.2)--(.5,-1.1);
	
	\begin{scriptsize}
	\node(a)at(1.6,-.65){base change};
	\end{scriptsize}

	\draw[thin] (0,0)--(0,2.75);
	\draw[thin,dashed] (0,2.75)--(0,3.4);
	\draw[thin] (0,.75)--(1,.2);
	\draw[thin] (0,2.75)--(1.75,2.5);
	\draw[thin] (0,1.75)--(2,1.25);
	\draw[thin] (0,1.75)--(-2,1.25);

	\draw[fill] (0,0)circle(1pt);
	\draw[fill] (1,.2)circle(1pt);
	\draw[fill] (2,1.25)circle(1pt);
	\draw[fill] (-2,1.25)circle(1pt);
	\draw[fill] (1.75,2.5)circle(1pt);

	\begin{scriptsize}
	\node(a)at(-.2,.05){$0$};
	\node(a)at(1.25,.1){$\pi$};
	\node(a)at(-2.7,1.25){$-\sqrt{\pi}$};
	\node(a)at(2.5,1.25){$\sqrt{\pi}$};
	\node(a)at(2,2.5){$1$};
	\end{scriptsize}
	
	\begin{tiny}
	\node(a)at(-.3,.75){$(1)$};
	\node(a)at(-.3,2.12){$(1)$};
	\node(a)at(-.3,2.75){$(1)$};
	\end{tiny}

	\draw[fill=darkblue, opacity=.2] plot [smooth cycle, tension=1.5] coordinates {(0,1.75) (1.65,1.7) (1.65,1)};
	
	\draw[fill=darkblue, opacity=.2] plot [smooth cycle, tension=1.5] coordinates {(0,1.75) (-1.65,1.7) (-1.65,1)};

	\end{scope}

	\begin{scope}[xshift=8cm]

	\draw[thin] (0,0)--(0,2.75);
	\draw[thin,dashed] (0,2.75)--(0,3.2);
	\draw[thin] (0,.75)--(.5,.2);
	\draw[thin] (0,1.75)--(2,1.25);
	\draw[thin] (1.5,1.375)--(2,1.65);
	\draw[thin] (0,2.75)--(.75,2.5);
	\draw[very thick] (0,1.75)--(1.5,1.375);

	\draw[fill] (0,0)circle(1pt);
	\draw[fill] (.5,.2)circle(1pt);
	\draw[fill] (2,1.25)circle(1pt);
	\draw[fill] (2,1.65)circle(1pt);
	\draw[fill] (.75,2.5)circle(1pt);

	\begin{scriptsize}
	\node(a)at(-.2,.05){$0$};
	\node(a)at(.7,.2){$\pi$};
	\node(a)at(2.5,1.15){$[\sqrt{\pi}]$};
	\node(a)at(2.7,2){$[\sqrt{\pi}+\pi^{n/2}]$};
	\node(a)at(.9,2.3){$1$};
	\end{scriptsize}
	
	\begin{tiny}
	\node(a)at(-.3,.75){$(1)$};
	\node(a)at(-.3,1.75){$(2)$};
	\node(a)at(-.3,2.75){$(1)$};
	\end{tiny}

	\draw[fill=red, opacity=.2] plot [smooth cycle, tension=1.5] coordinates {(1.5,1.375) (-.65,3) (-.65,-.2)};
	
	\end{scope}
	
	\begin{scope}[xshift=8cm,yshift=4.2cm]
	
	\draw[thin,->] (.75,-.2)--(.75,-1.1);
	
	\begin{scriptsize}
	\node(a)at(1.85,-.65){base change};
	\end{scriptsize}

	\draw[thin] (0,0)--(0,2.75);
	\draw[thin,dashed] (0,2.75)--(0,3.4);
	\draw[thin] (0,.75)--(.5,.2);

	\draw[thin] (0,1.75)--(2,1.25);
	\draw[thin] (1.5,1.375)--(2,1.65);
	\draw[thin] (0,2.75)--(.75,2.5);
	\draw[very thick] (0,1.75)--(1.5,1.375);
	
	\draw[thin] (0,1.75)--(-2,1.25);
	\draw[thin] (0,1.75)--(2,1.25);
	\draw[thin] (-1.5,1.375)--(-2,1.65);
	\draw[very thick] (0,1.75)--(-1.5,1.375);
	%
	%

	\draw[fill] (0,0)circle(1pt);
	\draw[fill] (.5,.2)circle(1pt);
	\draw[fill] (2,1.25)circle(1pt);
	\draw[fill] (2,1.65)circle(1pt);
	\draw[fill] (-2,1.25)circle(1pt);
	\draw[fill] (-2,1.65)circle(1pt);
	\draw[fill] (.75,2.5)circle(1pt);

	\begin{scriptsize}
	\node(a)at(-.2,.05){$0$};
	\node(a)at(0.7,.3){$\pi$};
	\node(a)at(2.4,1.15){$\sqrt{\pi}$};
	\node(a)at(2.75,2){$\sqrt{\pi}+\pi^{n/2}$};
	\node(a)at(-2.6,1.15){$-\sqrt{\pi}$};
	\node(a)at(-2.55,1.9){$-\sqrt{\pi}+\pi^{n/2}$};
	\node(a)at(.9,2.3){$1$};
	\end{scriptsize}
	
	\begin{tiny}
	\node(a)at(-.3,.5){$(1)$};
	\node(a)at(-.3,1.95){$(1)$};
	\node(a)at(-.3,2.75){$(1)$};
	\end{tiny}

	\draw[fill=red, opacity=.2] plot [smooth cycle, tension=1.5] coordinates {(1.53,1.7) (-.8,3.4) (-.8,.05)};
	
	\end{scope}

	\end{tikzpicture} 
	\caption{
	\label{figure:virtual_disc_annulus}	
	The bottom half of the figure depicts the virtual disc of Example~\ref{example:virtual_several_fractional_discs} with $m=2$ (on the left) and the virtual annulus with one end of Example~\ref{example_virtual_annulus_oneend} (on the right). 
	They are both trivialized by the base change to $K(\sqrt\pi)$, as shown in the top half of the figure.
	The numbers between parentheses denote the multiplicities of the nearby type 2 points, and the skeletons of the virtual annulus and of its trivialization are depicted in bold.}
\end{figure}

\pa{\label{point_tame_one-ended_annuli}
	More generally, let $X$ be a virtual annulus with one end, denote by $x$ the bending point of $X$, and suppose that $X$ is trivialized by a tamely ramified extension $L$ of $K$.
	Then we claim that the multiplicity of $m(x)$ divides the degree $[L:K]$ of $L$ over $K$, which has to be even by the existence of an element in $\Gal(L|K)$ of order 2.
	Then, thanks to Proposition~\ref{proposition_multiplicity_basechange} and in the terminology of \cite{FantiniTurchetti2018}, our claim is equivalent to the fact that the connected components of $X_L$ are annuli \emph{of even modulus}.
	If $L$ is a quadratic extension of $K$, our claim is contained in \cite[Theorem~8.3]{FantiniTurchetti2018}, and corresponds to the case shown on the bottom-right part of Figure~\ref{figure:virtual_disc_annulus}.
	In the general case, up to replacing $K$ with $\fraks(X)$ we can assume that $X$ is geometrically connected.
	Let us first show that $X$ can be embedded in $\pan$.
	The connected component $Y$ of $X\setminus\{x\}$ that intersects $\Sigma(X)$ nontrivially is a virtual annulus with two ends that over $L$ becomes a disjoint union of two fractional annuli.
	In particular, by \cite[Theorem~8.1]{FantiniTurchetti2018}, and as already discussed in~\ref{example_FantiniTurchetti2018}, $Y_\s$ is a fractional annulus over the field of constants $\fraks(Y)\cong K[T]/(T^2-\pi)$ of $Y$.
	It follows that $Y$ is a generalized fractional annulus, that is there exist constants $\alpha, \beta \in \Q$ such that $Y$ is isomorphic to the subspace
	\[
	Y' = \big\{y\in\aan \;\big|\; |\pi|^\beta < |(T^2-\pi)(y)| < |\pi|^\alpha\big\}
	\]
	of $\aan$.
	Set
	\[
	X'_\alpha = \big\{y\in\aan \;\big|\; |(T^2-\pi)(y)| < |\pi|^\alpha\big\}
	\]
	and
	\[
	X'_\beta = \big\{y\in\pan \;\big|\; |(T^2-\pi)(y)| > |\pi|^\beta\big\}.
	\]
	Both $X'_\alpha$ and $X'_\beta$ contain, adjacent to their unique end, the fractional annulus $Y'$.
	Now, if when approaching the bending point $x$ the function $|T^2-\pi|$ on $Y$ tends to $|\pi|^\alpha$, then $X$ and $X'_\beta$ can be glued along $Y \cong Y'$.
	On the other hand, if that function tends to $|\pi|^\beta$, then it is $X$ and $X'_\alpha$ that can be glued along $Y \cong Y'$.
	In both cases, we obtain a proper curve $X''$ of genus 0 that contains $X$ as an open subspace. 
	By \cite[Th\'eor\`eme 3.7.2]{Ducros} $X''$ is the analytification of an algebraic projective curve of genus 0 over $K$, that is, a conic.
	Since the residue field of $K$ is algebraically closed, there is a rational point on the special fiber of a suitable model of $X''$ that, thanks to Hensel's lemma, can be lifted to a $K$-rational point in $X''$.
	As a result, $X''\cong \pan$, and so $X \subset X''$ can be embedded in $\pan$.
	Since $Y$ is trivialized by a tamely ramified extension of $K$, up to an automorphism of $\pan$ we can assume that such an embedding identifies $Y\subset X$ with $Y'\subset \pan$.
	It follows that $X$ is identified either with $X'_\beta$ or with $X'_\alpha$.
	Since $X$ is geometrically connected, the latter is not possible, and therefore $X \cong X'_{\beta}$.
	In particular, it is now easy to see that the bending point $x$ of $X$ has multiplicity 2 (it is the same of the one in Figure~\ref{figure:virtual_disc_annulus}) and hence divides the degree $[L:K]$.
}

In the next sections we will often make use of the following three results in order to prove that some semi-affinoid curve is a virtual disc or a virtual annulus.
The first two are reformulations of results proved by Ducros in \cite{Ducros}, while the third is a consequence of the second one.

\Pa{Fusion Lemmas}{\label{lemmas_fusion}
	Let $X$ be a connected semi-affinoid open $K$-curve, and let $x$ be a point of $X$ of genus zero such that $X\setminus\{x\}$ is the disjoint union of two subspaces $X'$ and $X''$ and of virtual discs.
	Then:
	\begin{enumerate}
		\item\label{lemma_fusion_disc} if $X'$ is a virtual disc and $X''$ is a virtual annulus with two ends such that there exists a virtual annulus with two ends $X''' \subset X$ containing both $X''$ and $x$, then $X$ is a virtual disc, trivialized by any extension that trivializes $X'''$;

		\item\label{lemma_fusion_annulus} if $X'$ and $X''$ are virtual annuli with two ends and $\fraks(X')=\fraks(X'')=\fraks(x)$, then $X$ is itself a virtual annulus with two ends, trivialized by any extension of $K$ that trivializes both $X'$ and $X''$;

		\item\label{lemma_fusion_annulus_one_end} if $X'$ is a virtual annulus with one end and bending point $y$ and $X''$ is a virtual annulus with two ends such that $\fraks(X'\setminus\{y\})=\fraks(X'')=\fraks(x)$, then $X$ is itself a virtual annulus with one end, trivialized by any extension of $K$ that trivializes both $X'$ and $X''$.
	\end{enumerate}
}

\begin{proof}
	Part $(i)$ is \cite[Lemme 5.1.1]{Ducros} applied to $X'$ and $X'''$.
	Part $(ii)$ in the case $\s(X')=\s(X'')=K$ is an immediate consequence of \cite[Lemme 5.1.2]{Ducros}.
	We deduce the general case from the case above applied to $X_\fraks$, since under our hypotheses we have $\fraks(X)=\fraks(x)$.
	To see this equality, consider the base change $X_\Kalg$ of $X$ to $\Kalg$; then the inverse image of the skeleton $\Sigma(X)=\Sigma(X')\cup\{x\}\cup\Sigma(X'')$, which is a line segment, is a disjoint union of $s(x)$-many line segments, and his complement is a disjoint union of discs.
	Therefore, $X_\Kalg$ has itself $s(x)$-many connected components, and therefore $[\fraks(X):K]=s(x)$.
	To prove part $(iii)$, we can again suppose that $X'$ is geometrically connected by working over $\fraks(X')$, which is a subfield of $\fraks(X'')$.
	Then, if $L$ is an extension of $K$ that trivializes both $X'$ and $X''$ and $\tau\colon X_L \to X$ denotes the base change morphism, the base change $X''_L$ consists of two disjoint annuli, each connected to the annulus $X'_L$ by either one of the two points in $\tau^{-1}(x)$.
	By applying twice the second part we deduce that $X_L$ us an annulus.
\end{proof}

\begin{proposition}\label{proposition_Ducros-criterion}
	Let $C$ be the analytification of a smooth, projective, and connected algebraic $K$-curve and let $X$ be an open subspace of $C$ whose topological boundary consists of type 2 points.
	Then $X$ is a virtual disc (respectively a virtual annulus) if and only if the following conditions are met:
	\begin{enumerate}
		\item $X$ does not contain type 2 points of positive genus;
		\item every connected component of $X \times_K \Kalg$ is contractible;
		\item every connected component of $X \times_K \Kalg$ has precisely one (respectively two) ends.
	\end{enumerate}
\end{proposition}

\begin{proof}
	This follows immediately from {{\cite[Proposition 5.1.18]{Ducros}}} applied to every connected component of $X \times_K \Kalg$.
\end{proof}

\section{Structure of analytic curves and triangulations}

In this section we move to the study of the models of $C$ over the valuation ring $R$ of $K$.
We focus on the relation between models of $C$ and finite non-empty subsets of type 2 points of $C$ and begin the study of semi-stability.

\Pa{Models}{
Let $C$ be the analytification of a smooth, projective, and connected algebraic $K$-curve $\frakC$.
We call \emph{$R$-model} (or simply \emph{model}) of $C$ any flat, proper scheme $\calC$ over $R$ equipped with an isomorphism from its generic fiber to $\frakC$.
We denote by $\calC_k$ the special fiber $\calC \times_R k$ of $\C$.
If $\C$ and $\C'$ are two models of $C$, a morphism of $R$-schemes $f\colon \C'\to\C$ is said to be a \emph{morphism of models} of $C$ if it is compatible with the given identifications of the generic fibers with $\frakC$.
\\
A \emph{vertex set} of $C$ is a finite and nonempty set $V$ of points of type 2 of $C$.
}

\pa{\label{point_formal_fibers}
If $\calC$ is a model of $C$, the reduction modulo $\pi$ induces a natural surjective and anti-continuous map
\[
\Sp_\C \colon C \longrightarrow \C_k
\]
(that is, the inverse image of an open subset of $\C_k$ is closed in $C$), called the \emph{specialization map} (or \emph{reduction map}) associated with $\C$.
If $\C$ is normal and $\eta$ is the generic point of an irreducible component of $\C_k$, then $\Sp_\C^{-1}(\eta)$ consists of a single type 2 point of $C$.
If $P$ is a closed point of $\C_k$, then $\Sp_\C^{-1}(P)$ is a semi-affinoid space whose associated special $R$-algebra is the completion $\widehat{\calO_{\C,P}}$ of the local ring of the model $\C$ at $P$.
We call the open subset $\Sp_\C^{-1}(P)$ of $C$ the \emph{formal fiber} of $P$ in $C$.
These notions, as well as those surrounding the content of the following proposition, were first investigated in \cite{BoschLuetkebohmert1985} in the setting of rigid analytic geometry and are at the heart of the description of the structure of non-archimedean $K$-analytic curves.
For a more modern formulation we refer to \cite[Section 4]{Berkovich1990} and \cite{Ducros}, or to \cite{BakerPayneRabinoff2013} and \cite{Temkin2015} for the analogous result in the (simpler) case of an algebraically closed base field.
}

\begin{theorem}\label{theorem_models-vertex-sets}
	The correspondence 
	\[
	\C \longmapsto V_\C = \big\{\Sp_\C^{-1}(\eta)\,\big|\,\eta\mbox{ generic point of a component of }\C_k\big\},
	\]
	where $\Sp_{\calC} \colon C \to \calC_k$ is the specialization map, induces a bijection between the partially ordered set of isomorphism classes of normal (algebraic) $R$-models $\calC$ of $C$, ordered by morphisms of models, and the partially ordered set of vertex sets of $C$, ordered by inclusion.
	Moreover, $\Sp_{\calC}$ induces a bijection 
	\[
	\big\{\textrm{closed points of }\calC_k\big\} \stackrel{\sim}{\longrightarrow} \pi_0(C\setminus V_\C)
	\]
	via the inverse image $P \mapsto \Sp_{\calC}^{-1}(P)$.
\end{theorem}

The surjectivity of the correspondence, which is the only part of the statement which was not already implicit in the work of \cite{BoschLuetkebohmert1985}, can be deduced directly from \cite[Th\'eor\`eme 6.3.15 and Th\'eor\`eme 3.7.6]{Ducros}.
It can also be obtained starting from any model whose associated vertex set contains a given one, and proving that the additional components of the special fiber can be contracted while preserving algebraicity and normality of the model.

\Pa{Semi-stable models}{\label{point_semistable_models}
Recall that a model $\C$ of $C$ over $R$ is said to be a \emph{semi-stable model} of $C$ if its special fiber $\calC_k$ is reduced and has at worst double points as singularities.
It follows from the description of formal fibers of \ref{point_formal_fibers} that a model $\C$ of $C$ is a semi-stable model of $C$ if and only if the connected components of $C\setminus V_\C$ are all discs and annuli (see \cite[Propositions 2.2 and 2.3]{BoschLuetkebohmert1985}).
Indeed, a normal model of $C$ is semi-stable if and only if, for every point $P$ of the special fiber $\C_k$ of $\C$, there exists $e\in \N$ such that the completed local ring $\widehat{\calO_{\C,P}}$ is isomorphic to either $R[[S,T]]/(T-\pi^e)\cong R[[S]]$, if $P$ is a smooth point of $\C_k$ through $P$, or to $R[[S,T]]/(ST-\pi^e)$, if $P$ is a double point of $\C_k$ (see \cite[Corollary 10.3.22]{Liu2002} for a direct computation of the latter). 
}

\begin{remark}\label{remark_multiplicity_genus_base_change}
Under the correspondence of Theorem~\ref{theorem_models-vertex-sets}, the notions of multiplicity of a point of type 2 is consistent with its geometrical counterpart.
More precisely, if $\C$ is a model of $C$ and $x\in V_\C$, then $m(x)$ is the multiplicity of the $k$-curve $\overline{\Sp_{\calC}(x)}$ inside $\calC_k$ (see for example \cite[Proposition 6.5.2.(1)]{Ducros}).
This is not true for the genus.
Indeed, the function field of the curve $\overline{\Sp_{\calC}(x)}$ is exactly the residue field $\widetilde{\mathscr H(x)}$ of the complete residue field $\mathscr H(x)$ of $C$ at $x$ (see again \cite[Proposition 6.5.2.(2)]{Ducros}), but the genus $g(x)$ of $x$ is the (geometric) genus of a curve above $\overline{\Sp_{\calC}(x)}$ in the base change of the model $\C$ to the algebraic closure of $K$, which could be greater than the genus of $\overline{\Sp_{\calC}(x)}$ since passing to the special fiber does not commute with base change.
For example, the analytification $C$ of an elliptic curve $\frakC$ that has potentially good reduction has a type 2 point $x$ which has genus 1.
However, if $\frakC$ does not have good reduction over $K$, then it has a regular model $\C$ over $R$ whose special fiber consists of rational curves.
In particular, the point $x$ is associated with an irreducible component $E$ of the special fiber of a model obtained from the regular model $\C$ by a sequence of point blowups, which implies that $E$ is a rational curve.
\end{remark}

\begin{remark}
	The approach to models via vertex sets permits to avoid some difficulties related to the fact that when constructing new models, for example via base change, it is usually necessary to normalize the resulting scheme.
	Here are two examples of this phenomenon in action:
	\begin{enumerate}
		\item If $\C$ is any model of $C$, the inverse image $\Sp_\C^{-1}(\eta)$ through the specialization map of the generic point $\eta$ of a component $E$ of $\C_k$ consists of finitely many type 2 points of $C$, one for each component above $E$ in the special fiber of the normalization of $\C$.
		
		\item Let $L$ be a finite extension of $K$ and denote by $\tau\colon C_L\to C$ the base change morphism.
		Let $V$ be a vertex set of $C$, let $\C_V$ be the model of $C$ associated with $V$ and consider the model $\C_{\tau^{-1}(V)}$ of $C_L$ associated with the vertex set $\tau^{-1}(V)$ of $C_L$.
		Then $\C_{\tau^{-1}(V)}$ is the normalization of the base change of $\C_V$ to the valuation ring of $L$.
	\end{enumerate}
\end{remark}

\pa{
	Let $X$ be a semi-affinoid open subspace of $C$ whose boundary $\partial X$ in $C$ consists of type 2 points.
	Then there exists a natural surjective map $\partial\colon\Ends(X)\to\partial X\subset C$, which sends an end $\varepsilon$ represented by a family $(U_n)_n$ to the type 2 point $\partial(\varepsilon)=\cap_n \partial(U_n)$ of $C$.
	As observed in \ref{point_splitting_function}, we have $\fraks\big(\partial(\varepsilon)\big)\subset\fraks(\varepsilon)$.
	\\
	The multiplicity $m\big(\partial(\varepsilon)\big)$ of the type 2 point $\partial(\varepsilon)$ of $C$ only depends on $X$ and $\varepsilon$, while it does not depend on the ambient curve $C$ itself.
	In order to see this, one can reason as follows.
	Consider a regular model $\C$ of $C$ with strict normal crossings special fiber such that $\partial(\varepsilon)$ belongs to the vertex set $V$ associated with $\C$ and that $X$ contains the unique connected component $X'$ of $C\setminus V$ satisfying $\varepsilon\in \Ends(X')$.
	Then the multiplicity $m\big(\partial(\varepsilon)\big)$ of $\partial(\varepsilon)$ can be read from the completed local ring of the point $\Sp_\C(X')$ of the special fiber $\C_k$ of $\C$, as it is the multiplicity in $\C_k$ of the irreducible component associated with $\varepsilon$, and so also from the canonical model of the regular semi-affinoid space $X'$, and can therefore be read in $X$ without the need of considering the geometry of $C$.
	The multiplicity $m\big(\partial (\varepsilon)\big)$ will be denoted by $m(\varepsilon)$ and called \emph{multiplicity of the end} $\varepsilon$.
	For example, the multiplicity of any end of a disc or an annulus is equal to 1.
	\\
	Note that, for the same reason discussed in Remark~\ref{remark_multiplicity_genus_base_change}, it is not possible to define intrinsically the genus of an end, as a semi-affinoid space can generally be realized as a formal fiber of points lying on curves of different genera.
	For example, if $C$ is the analytification of a genus $g$ curve with good reduction, so that it contains a point $x$ of genus $g$, any connected component $X$ of $C\setminus\{x\}$ is a disc whose boundary point has genus $g$.
	Observe also that one could define directly the multiplicity $m(\varepsilon)$ of an end $\varepsilon$ of a semi-affinoid space $X$ as the multiplicity of the irreducible component corresponding to $\varepsilon$ in the special fiber of a snc resolution of the canonical model of $X$.
	However, the notion of irreducible components of special formal schemes and their multiplicities is quite delicate (see \cite[Sections 2.4 and 2.5]{Nicaise2009}) and its treatment goes beyond the scope of our work.
}

\Pa{Triangulations}{\label{point:def_triangulation}
A vertex set $V$ of $C$ is said to be a \emph{triangulation} of $C$ if each connected component of $C\setminus V$ is either a virtual disc or a virtual annulus.
	If moreover all virtual annuli among the connected components of $C\setminus V$ have two ends, $V$ is said to be a \emph{strong triangulation} of $C$.
Any strong triangulation of $C$ is a triangulation of $C$ by definition.
Conversely, given a triangulation $V$ of $C$, thanks to the observation of \ref{point_break_one-ended_annuli}, by adding to $V$ the bending point of each connected component of $C\setminus V$ which is a virtual annulus with one end we obtain a strong triangulation of $C$; it is the minimal strong triangulation of $C$ containing $V$.
}

\Pa{Terminology}{\label{point:def_triangulation_terminology}
	Since our definition of virtual annulus is more general than the one of \cite{Ducros}, allowing for virtual annuli with one end, the same goes with our definition of triangulations.
	The triangulations of \emph{loc. cit.} are what we call strong triangulations here.
}

\begin{proposition}\label{proposition_existence_triangulation}
		The curve $C$ admits a triangulation.
\end{proposition}

\begin{proof}
By the semi-stable reduction theorem there exist a finite Galois extension $L$ of $K$ and a semi-stable model $\C$ of $\frakC_L$ over the valuation ring of $L$.
Denote by $\tau\colon C_L\to C$ the base change morphism and by $V\subset C_L$ the vertex set of $C_L$ associated with $\C$.
We can assume without loss of generality that the natural action of $G=\Gal(L|K)$ on $C_L$ extends to $\C$ (from the point of view of vertex sets this is obtained by replacing $V$ with its closure under the $G$-action, which still yields a semi-stable model of $C_L$).
For each connected component $X$ of $C\setminus \tau(V)$, the base change $X_L=\tau^{-1}(X)$ of $X$ to $L$ is then a finite union of isomorphic connected components of $C_L\setminus V$, and hence $X$ is a virtual disc or annulus.
\end{proof}

\begin{remark}
Using the results of the next section we can prove a more precise result, namely that $V$ is a triangulation of $C$ if and only if there exists a finite extension $L$ of $K$ trivializing all the components of $C\setminus V$ simultaneously, in which case the model of $C_L$ associated with the vertex set $\tau^{-1}(V)$, where $\tau\colon C_L\to C$ is the base change morphism, is semi-stable.
The reader should also note that, while for our purposes it was sufficient to deduce the result of Proposition~\ref{proposition_existence_triangulation} from the existence of a semi-stable model of $\frakC$, it is also possible to obtain it via a delicate study of the geometry of the analytic curve $C$ itself, and therefore use it to prove the semi-stable reduction theorem (see \cite[Section 6.4]{Ducros}).
\end{remark}

\begin{remarks}\label{remarks_basic_properties_triangulations} We discuss here some simple properties of triangulations:
	\begin{enumerate}
		\item Any triangulation of $C$ contains the set of type 2 points of $C$ that have positive genus. 
		Indeed, virtual discs and annuli can be embedded in the analytic affine line after a base change, and the genus of a type 2 point is invariant under base change.
		\item A triangulation $V$ of $C$ gives rise to a graph $\Sigma(V)$ embedded in $C$, called the \emph{skeleton of $C$ associated with $V$} and defined as the graph whose set of vertices is $V$ and whose edges are the skeletons of the virtual annuli in $\pi_0(C\setminus V)$.
		Since $C$ is the analytification of an algebraic curve, only finitely many of the elements of $\pi_0(C\setminus V)$ are virtual annuli, therefore $\Sigma(V)$ is a finite graph.
		\item If $V$ is a triangulation of $C$, then $C$ retracts by deformation onto $\Sigma(V)$, as each virtual disc retracts onto its boundary point and each virtual annulus retracts onto its skeleton.
		\item \label{remark_genus_formula}
		Assume that $\frakC$ has a semi-stable model $\C$ over $R$.
		Then, for every triangulation $V$ of $C$ we have
		\[
		g(\frakC) = b\big(\Sigma(V)\big)+\sum_{x\in V}g(x),
		\]
		where $b\big(\Sigma(V)\big)=$ denotes the first Betti number of $\Sigma(V)$.
		Indeed, if $\Gamma$ is the dual graph of the special fiber of $\C$ then $C$ retracts by deformation also onto a copy of $\Gamma$, and thus we have $b\big(\Sigma(V)\big)=b(\Gamma)$. The formula then follows from the standard result \cite[Lemma 10.3.18]{Liu2002}.
	\end{enumerate}
\end{remarks}

\section{Minimal triangulations and ramification}

In this section we discuss the notion of minimal (strong) triangulations and prove our first main result, Theorem~\ref{theorem_main}.
We restrict ourselves to the case of curves of genus greater than 1; we will explain how to adapt the proofs to the case of elliptic curves in Section~\ref{section_elliptic_curves}.

\pa{We say that a vertex set $V$ of $C$ is a \emph{minimal triangulation} of $C$ (respectively, a \emph{minimal strong triangulation} of $C$) if it is minimal among the triangulations (respectively, among the strong triangulations) of $C$ ordered by inclusion. 
We can now state the main result of the section.
The existence (and uniqueness) of a minimal triangulation and of a minimal strong triangulation for curves of genus at least 2 will be proven in Proposition~\ref{proposition_mintr_nodes}.
}

\begin{theorem}\label{theorem_main}
	Assume that $g(C)>1$, let $V_{\mathrm{min-tr}}$ be a minimal triangulation of $C$, and let $L$ be a finite Galois extension of $K$ such that $\frakC_L$ has semi-stable reduction.
	Then:
	\begin{enumerate}
		\item \label{theorem_main_divisibility} $\lcm\{m(x)\,|\,x\in V_{\mathrm{min-tr}}\}\,\Big|\,[L:K]\,;$
		\item \label{theorem_main_effective} $L$ can be taken to be tamely ramified over $K$ if and only if $p$ does not divide $\lcm\{m(x)\,|\,x\in V_{\mathrm{min-tr}}\}$.
		When this is the case, then the minimal extension $L'$ of $K$ such that $\frakC_{L'}$ has semi-stable reduction is the unique totally ramified extension of $K$ of degree exactly $\lcm\{m(x)\,|\,x\in V_{\mathrm{min-tr}}\}$.
	\end{enumerate}
\end{theorem}

\pa{In order to prove this theorem we need some intermediate results, starting with an alternative description of minimal triangulations.
As we did in the case of virtual annuli, we call \emph{analytic skeleton} of $C$ the set $\anskC$ of those points of $C$ that have no neighborhood in $C$ isomorphic to a virtual disc.
If $V$ is a triangulation of $C$, then $\anskC$ is contained in the skeleton $\Sigma(V)$ associated with $V$ constructed in Remark~\ref{remarks_basic_properties_triangulations}, hence $\anskC$ is itself a finite graph.
We say that a point $x$ of $\anskC$ is a \emph{node of} $\anskC$ if either one of the following three conditions holds: 
\begin{enumerate}
\item the genus $g(x)$ of $x$ is positive;
\item the point $x$ has degree different from $2$ in the graph $\anskC$;
\item the splitting function $s\colon \anskC\to \Z$ is discontinuous (that is, not locally constant) at $x$.
\end{enumerate}

\noindent
The first part of the following proposition can also be found in the discussion of \cite[\S 5.4]{Ducros} (the precise statement is given in 5.4.12.1 of \emph{loc.\ cit.}); we include a direct proof for the reader's benefit, as it provides the starting point to prove the second part as well.
}

\begin{proposition}\label{proposition_mintr_nodes}
	Assume that $g(C)>1$.
	Then: 
	\begin{enumerate}
		\item $C$ admits a unique minimal strong triangulation $V_\mathrm{min-str}$, which coincides with the set of nodes of $\anskC$ and with the set of points of $C$ that have no neighborhood isomorphic to a virtual annulus with two ends;
		\item $C$ admits a unique minimal triangulation $V_\mathrm{min-tr}$, which coincides with the set of points of $C$ that have no neighborhood isomorphic to a virtual annulus.
	\end{enumerate}
\end{proposition}

\begin{proof}
In order to prove the first part of the proposition, we begin by showing that $\anskC$ is nonempty and that it has at least a node.
More precisely, we will show the existence of a node that has either positive genus or degree at least 2 in $\anskC$.
To do this, we can assume without loss of generality that $C$ does not have a type 2 point of positive genus.
Let $L$ be a finite Galois extension of $K$ and let $\tau\colon C_L\to C$ be the base change morphism.
Then we have $\ansk{C_L} \subset \tau^{-1}\big(\anskC\big)$.
Indeed, given $x \in C \setminus \ansk{C}$ there exists a virtual disc $X$ in $C$ that contains $x$, and thus $X_L\cong\tau^{-1}(X)$ is a disjoint union of virtual discs in $C_L$, from which the claim follows.
Now, using the semi-stable reduction theorem we can suppose that $\frakC$ has semi-stable reduction over $L$.
Then $C_L$ contains at least two loops, as can be deduced from the genus formula recalled in Remark~\ref{remark_genus_formula}.
Observe that if $\gamma\colon [0,1] \to C_L$ defines a loop on $C_L$, that is $\gamma$ is continuous, injective, and $\gamma(0)=\gamma(1)$, then $\gamma([0,1])$ is contained in $\ansk{C_L}$, because a point of $\gamma([0,1])$ cannot have a contractible open neighborhood with only one end.
In particular, $\ansk{C_L}$ is nonempty, and therefore so is $\tau\big(\ansk{C_L}\big)=\ansk{C_L}/\Gal(L|K)\subset \anskC$.
We will now show the existence of a node of degree at least 2 in $\anskC$.
This is immediate if $\tau\big(\ansk{C_L}\big)$ contains at least one point of degree at least three, (for example, this is the case if $\tau\big(\ansk{C_L}\big)$ contains at least two loops, that is if $\Gal(L|K)$ fixes point-wise at least two loops of $\ansk{C_L}$).
Whenever this is not the case, then $\tau\big(\ansk{C_L}\big)$ is either a loop or a compact interval.
In both cases, it is easy to verify by elementary topological arguments that there exists a point of degree 2 of $\tau\big(\ansk{C_L}\big)$ that has less inverse images under $\tau$ than some of its neighbors, so that the splitting function $\fraks$ is discontinuous at this point.
Let us now prove that the set $S$ of nodes of $\anskC$ is contained in all strong triangulations of $C$. 
Let $V$ be an arbitrary strong triangulation of $C$, which exists by Proposition~\ref{proposition_existence_triangulation}.
We already observed that $\anskC$ is contained in the skeleton $\Sigma(V)$ associated with $V$.
Assume that $x$ is a point of $V$ that has degree one in the graph $\Sigma(V)$ and such that $x$ is not in $\anskC$; such a point exists unless $V$ is contained in $\anskC$.
Then by the first Fusion Lemma~\ref{lemma_fusion_disc} we can glue the unique virtual annulus with two ends among the components of $C\setminus V$ adjacent to $x$ to a virtual disc containing $x$, showing that $V\setminus\{x\}$ is also a strong triangulation.
This ensures that we can delete recursively all vertices of degree one that are not contained in $\anskC$, until we obtain a strong triangulation $V'\subset V$ such that $\Sigma(V')\subset \anskC$, and hence  $\Sigma(V')=\anskC$.
If $x$ is a point of $S\setminus V'$, then the connected component $X$ of $C\setminus V$ containing $x$ is a virtual annulus whose skeleton is $X\cap\anskC$, since the skeleton of a virtual annulus coincides with its analytic skeleton.
In particular, $x$ has genus $0$, degree $2$ in $\anskC$, and $\fraks$ is constant on $\anskC$ locally at $x$, contradicting the fact that $x$ is a node of $\anskC$.
This proves that $S\subset V'$, in particular $S$ is contained in every strong triangulation of $C$.
To prove that $S$ is the unique minimal strong triangulation of $C$, it remains to show that $S$ is itself a strong triangulation.
In order to do so, we will show that we can delete any point $x$ of $V\setminus S$ and still obtain a strong triangulation.
Since such a point $x$ is not a node of $\anskC$, it has genus 0 and degree 2 in $\anskC$ and hence there are precisely two virtual annuli with two ends $X_1$ and $X_2$ among the connected components of $C\setminus V'$ adjacent to $x$, and the skeleton of each of the $X_i$ coincides with $X_i\cap\anskC$.
The fact that $x$ is not a node implies that $\fraks$ is constant on a neighborhood of $x$.
Since on the skeleton of a virtual annulus $\fraks$ is constant by \ref{skeletons_virtual_annuli}, this implies that $\fraks(X_1)=\fraks(X_2)$.
Then by the second Fusion Lemma~\ref{lemma_fusion_annulus} the connected component of $C\setminus\big(V'\setminus\{x\}\big)$ containing $x$ is a virtual annulus with two ends, therefore $V\setminus\{x\}$ is a strong triangulation of $C$.
This proves that $S$ is the minimal strong triangulation of $C$.
Now, any point of $C\setminus S$ has a neighborhood isomorphic to a virtual annulus with two ends, since this is the case for every point of a virtual disc.
Conversely, no point of $S$ can have a neighborhood isomorphic to a virtual annulus with two ends, otherwise by repeating the argument using the Fusion Lemma above we would  obtain a smaller strong triangulation.
This completes the proof of the first part of the proposition.
It remains to show that the set $S'$ obtained by removing from $S$ those points that have a neighborhood isomorphic to a virtual annulus is the minimal triangulation of $C$.
First we observe that if $x$ is a point of $S\setminus S'$, then any neighborhood as above will be a virtual annulus with one end whose bending point is $x$.
In particular, $x$ has degree 1 in $\anskC$ and genus $g(x)=0$.
It follows that $S'$ is nonempty, since we have proven that $\anskC$ has a node that has either positive genus or degree at least 2.
Observe also that $S'$ is contained in any triangulation $V''$ of $C$, since any point of $C\setminus V''$ has a neighborhood isomorphic to a virtual disc.
To conclude the proof of the proposition, it remains to show that $S'$ is itself a triangulation of $C$.
This can be done in a similar way as in the proof of the first part of the proposition, by removing the points of $S\setminus S'$ one by one applying the third Fusion Lemma~\ref{lemma_fusion_annulus_one_end}.
\end{proof}

\begin{remarks}
	Let us comment on a few points that arose in the course of the proof of Proposition~\ref{proposition_mintr_nodes}.
	\begin{enumerate}
		\item
		We have shown that if $C$ has genus at least two then $\anskC$ is nonempty and it has at least a node.
		Note that the hypothesis on the genus of $C$ is necessary: the projective line has empty analytic skeleton, and the analytic skeleton of an elliptic curve with multiplicative reduction is a circle with no nodes.
		However, the theory of this paper can be easily adapted to the general case once we take into account marked points in the definition of a minimal (strong) triangulation. 
		This will be discussed in Section~\ref{section_elliptic_curves}, where we will prove an analogue of Theorem~\ref{theorem_main} for elliptic curves.
		\item In the proof we have made use of the inclusion of $\ansk{C_L}$ in $\tau^{-1}\big(\anskC\big)$.
		In fact, the equality $\ansk{C_L} = \tau^{-1}\big(\anskC\big)$ holds, as can be seen from combining the proposition with Lemma~\ref{lemma_minimal_triangulation_basechange} below.
		\item It is also worth stating explicitly the following fact that appeared in the proof: an element of $V_\mathrm{min-str}$ which is not an element of $V_\mathrm{min-tr}$ has necessarily degree 1 in $\anskC$.
	\end{enumerate}
\end{remarks}

\begin{lemma}\label{lemma_minimal_triangulation_basechange}
	Assume that $g(C)>1$, let $V$ be a vertex set of $C$, let $L$ be a finite extension of $K$, and denote by $\tau\colon C_L\to C$ the base change morphism.
	Then:
	\begin{enumerate}
		\item $V$ is a triangulation of $C$ if and only if  $\tau^{-1}(V)$ is a triangulation of $C_L$.
		\item $V$ is the minimal triangulation of $C$ if and only if  $\tau^{-1}(V)$ is the minimal triangulation of $C_L$.
	\end{enumerate}
\end{lemma}

\begin{proof}
	Observe that if $X\subset C$ is a virtual disc (respectively a virtual annulus), then all the connected components of $X_L\cong\tau^{-1}(X)$ are virtual discs (respectively virtual annuli).
	In particular, if $V$ is a triangulation of $C$ then $\tau^{-1}(V)$ is a triangulation of $C_L$.
	Suppose now that $L$ is Galois over $K$ and let $Y$ be a connected component of $C_L \setminus \tau^{-1}(V)$.
	Then the connected component of $\tau^{-1}\big(\tau(Y)\big)$ that contains $Y$ is $Y$ itself, and in particular all connected components of $\tau^{-1}\big(\tau(Y)\big)$ are translates of $Y$ under some element of $\Gal(L|K)$, and are thus isomorphic to $Y$.
	We deduce that if $Y$ is a virtual disc (respectively a virtual annulus), then so is $\tau(Y)$, and in particular $V$ is a triangulation as long as $\tau^{-1}(V)$ is one.
	In the general case, we obtain the same result by first doing a base change to a finite Galois extension of $K$ containing $L$ and the applying the result we have just proven in the Galois case.
	This concludes the proof of the first part of the lemma.
	If $L$ is Galois over $K$, the second part of the lemma follows now from the fact that the minimal triangulation of $C_L$ is Galois-invariant, which is an immediate consequence of Proposition~\ref{proposition_mintr_nodes}.
	Indeed, a point of $C_L$ has a neighborhood isomorphic to a virtual annulus if and only if the same is true for any of its translates by elements of $\Gal(L|K)$.
	The general case can again be reduced to the Galois case in a similar way as in the proof of the first part of the lemma.
\end{proof}

\begin{remark}
	The main reason why the notion of triangulation is more useful than the one of strong triangulation in our context is the fact that it behaves better under base change.
	Indeed, the lemma above does not hold for the strong triangulations of $C$, since any strong triangulation is required to contain the set of bending points of the virtual annuli with one end that are contained in $C$, while the base change of that set under an extension $L$ of $K$ trivializing those virtual annuli is not necessarily contained in all strong triangulations of $C_L$.
\end{remark}

\begin{proposition}\label{proposition_mintr_semistable_extension}
	Assume that $g(C)>1$ and let $V_{\mathrm{min-tr}}$ be the minimal triangulation of $C$. 
	Then the following are equivalent:
	\begin{enumerate}
		\item $\mathfrak C$ admits semi-stable reduction over $K$\,;
		\item all the connected components of $C\setminus V_{\mathrm{min-tr}}$ are discs or annuli over $K$\,;
		\item $m(x)=1$ for every $x$ in $V_\mathrm{min-tr}$.
	\end{enumerate}
	When these properties hold, then $V_{\mathrm{min-tr}}$ is a strong triangulation of $C$ and the associated model is the stable model of $\mathfrak C$.
\end{proposition}

\begin{proof}
	Denote by $\C$ the model of $C$ associated with its minimal triangulation $V_\mathrm{min-tr}$.
	As discussed in \ref{point_semistable_models}, part $(ii)$ implies that $\C$ is semi-stable, and thus $(i)$ holds.
	The implication $(i)\implies(ii)$ follows from the fact that there exists a triangulation $V$ whose associated model is semi-stable, so that $V$ has to contain $V_{\mathrm{min-tr}}$ by minimality of the latter, and from the following claim: if $X$ is a virtual disc (or a virtual annulus) that can be decomposed as a disjoint union of finitely many type 2 points, finitely many annuli, and some discs, then $X$ is itself a disc (respectively an annulus).
	If $X$ is a virtual disc or a virtual annulus with two ends, the claim is a simple consequence of the Fusion Lemmas~\ref{lemmas_fusion}.
	On the other hand, if $X$ is a virtual annulus with one end admitting such a decomposition, there is an annulus $X' \subset X$ whose skeleton intersects the interior of the skeleton of $X$.
	This gives rise to a contradiction, since the splitting function $x\mapsto s(x) =  [\fraks(x):K]$ of \ref{point_splitting_function} is equal to 1 on the skeleton of the annulus $X'$, and is greater than 1 on the the interior of the skeleton of $X$, therefore proving the claim and thus establishing $(ii)$.
	In particular, this shows that if $\frakC$ admits semi-stable reduction over $K$ then $\C$ is a semi-stable model of $C$.
	Since the vertex set of any semi-stable model of $C$ is a triangulation, we deduce that $\C$ is the minimal semi-stable model of $C$, that is its stable model.
	Moreover, since no component of $C\setminus V_\mathrm{min-tr}$ is a virtual annulus with one end, $V_\mathrm{min-tr}$ is also a strong triangulation of $C$.
	The implication $(ii)\implies(iii)$ is immediate, since the boundary points of discs and annuli have multiplicity 1.
	In order to prove the implication $(iii)\implies(i)$, denote again by $\C$ the normal model of $C$ associated with $V_\mathrm{min-tr}$, let $L$ be a finite extension of $K$ such that $\frakC_L$ admits semi-stable reduction over $L$, and let $\C'$ be the base change of $\C$ to the valuation ring of $L$.
	Note that the special fiber $\C'_k$ of $\C'$ coincides with that the special fiber $\C_k$ of $\C$, since we have that $\C'_k=\calC'\times_{R'} k = \calC\times_R R' \times_{R'} k  =\calC_k$.
	To prove that $\C_k$ is reduced, we make use of Serre's criterion as follows.
	Since $m(x)=1$ for every $x$ in $V_\mathrm{min-tr}$, then by Remark \ref{remark_multiplicity_genus_base_change} the local ring at the generic point of every irreducible component of the special fiber is reduced, and therefore it is regular.
	In particular, the scheme $\C_k$ satisfies the property $R_0$.
	Moreover, since $\C_k$ is a Cartier divisor inside the model $\C$, and this model satisfies the property $S_2$ by normality, then $\C_k$ satisfies the property $S_1$.
	It follows that the special fiber $\C'_k$ of $\C'$, being isomorphic to $\C_k$, is reduced as well, and so $\C'$ is itself normal by~\cite[Lemma 4.1.18]{Liu2002}.
	We deduce that $\C'$ is the model of $C_L$ associated with the vertex set $\tau^{-1}(V_\mathrm{min-tr})$, which is the minimal triangulation of $C_L$ by Proposition~\ref{proposition_mintr_semistable_extension}, and therefore $\C'$ is the stable model of $C_L$ by the previous part of the proof.
	Since the definition of stable model only depends on the geometry of the special fiber, and since $\C_k\cong\C'_k$, we deduce that $\C$ is itself stable, proving $(i)$.
\end{proof}

\begin{proof}[Proof of Theorem~\ref{theorem_main}]
	By combining Proposition~\ref{proposition_mintr_semistable_extension} and Lemma~\ref{lemma_minimal_triangulation_basechange} we deduce that, if $x$ is a point of $V_{\mathrm{min-tr}}$ and $y$ is a point of $\tau^{-1}(x)$, then $m(y)=1$.
	Then Proposition~\ref{proposition_multiplicity_basechange} implies that $\gcd\{m(x),[L:K]\}$ must be equal to $m(x)$, that is $m(x)$ divides $[L:K]$, which proves $(i)$.
	We move now to the proof of part $(ii)$.
	Set $d=\lcm\{m(x)\,|\,x\in V_{\mathrm{min-tr}}\}$, and observe that it follows from $(i)$ that if $L$ can be taken to be tamely ramified over $K$ then $p$ does not divide $d$. 
	Conversely, assume that $p$ does not divide $d$, let $L'$ be the unique totally ramified extension of $K$ of degree $d$, and denote by $\tau'\colon C_{L'}\to C$ the corresponding base change morphism.
	Since $L'$ is tamely ramified over $K$, we deduce from Proposition~\ref{proposition_multiplicity_basechange} that $m(y)=1$ for every $y$ in $\tau'^{-1}(V_\mathrm{min-tr})$.
	But by Lemma~\ref{lemma_minimal_triangulation_basechange} the vertex set $\tau'^{-1}(V_\mathrm{min-tr})$ is the minimal triangulation of $C_{L'}$.
	It then follows again from Proposition~\ref{proposition_mintr_semistable_extension} that $\C_{L'}$ admits semi-stable reduction over $L'$, which is what we wanted to show.
\end{proof}

\pa{\label{point:lcm_equal_in_tame_case}
	Assume that, in the notation of Theorem~\ref{theorem_main}, the minimal extension $L$ of $K$ such that $\frakC_L$ has semi-stable reduction is tamely ramified. 
	Then, it follows immediately from~\ref{point_tame_one-ended_annuli} that the multiplicity of the bending point of any virtual annulus with one end among the components of $C\setminus V_\mathrm{min-tr}$ divides the degree $[L:K]$, which coincides with $\lcm\{m(x)\,|\,x\in V_{\mathrm{min-tr}}\}$ by Theorem~\ref{theorem_main_effective}.
	In particular, we deduce that
	\(
	\lcm\{m(x)\,|\,x\in V_{\mathrm{min-tr}}\} = \lcm\{m(x)\,|\,x\in V_{\mathrm{min-str}}\}\,.
	\)
}

\section{Snc models and vertex sets}
\label{section_snc}

In this section, we introduce models with strict normal crossings, and prove the main technical result required to relate minimal snc models and minimal triangulations.

\Pa{Snc vertex sets}{
A vertex set $V$ of $C$ is called \emph{snc} if the model $\C_V$ associated with $V$ via the correspondence of Theorem \ref{theorem_models-vertex-sets} is a \emph{snc model} of $C$, that is, it is regular and its special fiber has strict normal crossings.
}

\pa{
With a snc vertex set $V$ of $C$, one can canonically associate a skeleton $\Sigma(V)$ inside $C$ even when $V$ is not a triangulation.
Indeed, each formal fiber of a double point of the special fiber $(\C_V)_k$ still contains a line segment connecting its two ends as a skeleton.
This can be done in a very intrinsic way and in arbitrary dimension as in \cite[Section 3.1]{MustataNicaise2015}, where it is also shown that $C$ retracts by deformation onto the skeleton $\Sigma(V)$.
}

Under suitable tame assumptions, we can show that a formal fiber of a snc model is a virtual disc or a virtual annulus.

\begin{proposition}\label{proposition_formalfibers}
	Let $\C$ be a snc model of $C$ and let $P$ be a closed point of the special fiber $\C_k$ of $\C$.
	Then:
	\begin{enumerate}
		\item if $P$ belongs to a unique irreducible component of $\C_k$ whose multiplicity $m$ is not divisible by $p$, then $\Sp_{\C}^{-1}(P)$ is a generalized fractional disc with field of constants $K(\pi^{\sfrac{1}{m}})$;
		\item if $P$ is a double point belonging to two irreducible components of $\C_k$ of multiplicities $m_1$ and $m_2$ and if $p$ does not divide the greatest common divisor $m=\gcd\{m_1, m_2\}$ of $m_1$ and $m_2$, then $\Sp_{\C}^{-1}(P)$ is a generalized fractional annulus with field of constants $K(\pi^{\sfrac{1}{m}})$.
	\end{enumerate}
\end{proposition}

\begin{proof}
Denote by $X$ the formal fiber $\Sp_{\C}^{-1}(P)$.
It is the semi-affinoid space generic fiber associated with the formal spectrum of the ring $\widehat{\calO_{\C, P}}$, which is the completion of the local ring of $\C$ at the point $P$.
A simple deformation-theoretic argument (see~\cite[Lemma 2.3.2]{ConradEdixhovenStein2003}, whose proof carries through in our case despite the fact that in the statement both multiplicities are supposed to be not divisible by $p$) shows that $\widehat{\calO_{\C, P}}$ is isomorphic to $R[[S,T]]/(S^m-\pi)$ in case $(i)$, and to $A\cong R[[S,T]]/(S^{m_1}T^{m_2}-\pi)$ in case $(ii)$.
In the first case, this shows that $X$ is precisely the virtual disc described in Example~\ref{example:virtual_several_fractional_discs}.
In the second case, observe that the field of constants $\s(X)$ of $X$ contains a $m$-th root  $\varpi=\pi^{\sfrac{1}{m}}$ of $\pi$, since if we write $m_i'=m_i/m$ for $i=1,2$, then $(S^{m_1'}T^{m_2'})^m=\pi$.
Then $X_\s$ is the semi-affinoid space associated with the special $R'$-algebra $R'[[S,T]]/(S^{m_1'}T^{m_2'}-\varpi)$, where we set $R'=R[\varpi]$.
In particular, we deduce that $R'[[S,T]]/(S^{m_1'}T^{m_2'}-\varpi)$ is the unique regular $R'$-algebra corresponding to a semi-affinoid space with two ends of multiplicities $m_1'$ and $m_2'$.
Since we prove in Corollary~\ref{corollary_existence_regular_fractional_annuli_coprime_multiplicities} that there exists a fractional annulus having these properties, it follows that $X_\s$ is a fractional annulus, which concludes the proof.
\end{proof}

Proposition \ref{proposition_formalfibers} enables us to say that the possible phenomena preventing a snc vertex set from being a strong triangulation are related to wild ramification.
One instance of this is the fact that type 2 points of positive genus whose associated residual curve is rational are either in the minimal snc vertex set or they have multiplicities divisible by the residue characteristic of $K$, as follows from the following lemma.

\begin{lemma}\label{lemma_curves_dreaming_of_having_positive_genus}
	Let $x$ be a type 2 point of $C$ such that $x$ belongs to the minimal strong triangulation of $C$ but not to its minimal snc vertex set.
	Then the residue characteristic $p$ of $K$ divides the multiplicity $m(x)$ of $x$.
\end{lemma}

\begin{proof}
Let $\calC$ be the minimal snc model of $C$ and let $P$ be the closed point of the special fiber $\calC_k$ of $\C$ such that $x$ belongs to the formal fiber $\Sp_\C^{-1}(P)$ of $\calC$ over $P$.
Since $x$ is an element of the minimal strong triangulation of $C$, it cannot be contained in any virtual disc or virtual annulus with two ends, hence by Proposition~\ref{proposition_formalfibers} all the components of $\calC_k$ passing through $Q$ must have multiplicity divisible by $p$.
The minimal snc model of $C$ whose vertex set contains $x$ is obtained by a sequence of blowups of closed points of $\calC$ above $Q$.
In particular, the exceptional component of every such blowup has to have multiplicity divisible by $p$ (see \cite[Proposition 9.1.21.(b) and Theorem 9.3.8]{Liu2002}), hence this holds for the component corresponding to $x$, so that $p$ divides $m(x)$.
\end{proof}

\begin{remark}\label{remark_curves_dreaming_of_having_positive_genus}
	Assume that $x$ is point of type 2 in $C$ such that $\widetilde{\mathscr H(x)}$ has positive geometric genus, which (as observed in Remark~\ref{remark_multiplicity_genus_base_change}) means that the genus of the component $E_x$ corresponding to $x$ in the special fiber of any regular model $\C$ of $C$ such that $x\in V_\C$ is positive.
	Then by Castelnuovo criterion $E_x$ cannot be contracted without introducing singularities of the model.
	This shows that $x$ has to be contained in $V_\mathrm{min-snc}$.
	On the other hand, the points studied in the lemma above contain an interesting special class of residual curves, those that have genus zero but acquire positive genus after some normalized base change.
	The existence of such curves is part of what makes the study of semi-stable reduction in the wild case particularly complicated.
	We will discuss this further in Section~\ref{section_wild} (see Example~\ref{example:wild_bad_cusp}).
\end{remark}

\section{The tamely ramified case}

Let $C$ be a curve that acquires semi-stable reduction after a tamely ramified extension of $K$.
In this section we describe explicitly the minimal strong triangulation of $C$ in terms of the combinatorics of its minimal snc model.
This is related to a classical result of T. Saito.

\pa{\label{point_def_principal}
As before, $C$ is the analytification of a proper, smooth algebraic $K$-curve $\mathfrak C$ of genus $g(\frakC)>1$.
Denote by $V_\mathrm{min-snc}$ the vertex set of $C$ associated with its minimal snc model.
Given a snc vertex set $V$ of $C$, we say that a point $x \in V$ is \emph{principal} if it satisfies either one of the following conditions:
\begin{enumerate}
\item the genus $g(x)$ of $x$ is positive;
\item the degree $\deg(x)$ of $x$ in the graph $\Sigma(V)$ is at least three;
\end{enumerate}
Given a snc vertex set $V$ of $C$, we call \emph{principalization} of $V$ the vertex set consisting of its principal points, and we denote it by $V_\mathrm{pr}$.
The fact that $V_\mathrm{pr}$ is nonempty follows from the argument already used in the course of the proof of Proposition~\ref{proposition_mintr_nodes}: since $g(\frakC)>1$ then either $C$ has a point or positive genus or it contains at least two loops, in which case $\Sigma(V)$ has at least a vertex of degree at least three by the genus formula recalled in~\ref{remark_genus_formula}.
}

\begin{remark}
	Observe that this definition of a principal vertex of $V$ differs from the usual definition of a principal component of the special fiber of the snc model $\C_V$ associated with $V$ (as for example given in \cite[Definition 6.1]{Halle2010}), because as we explained in Remark~\ref{remark_curves_dreaming_of_having_positive_genus} some type 2 points could have positive genus even when the associated residual curve is rational.
	Nevertheless, $V_\mathrm{pr}$ always contains the points associated with the principal components of $(\C_V)_k$.
\end{remark}

\pa{\label{point_intersection_multiplicities_chains}
	Let $V$ be a snc vertex set of $C$, let $X$ be a connected component of $C\setminus V_\mathrm{pr}$, and let $x_0\in V_\mathrm{pr}$ be a boundary point of $X$.
	Then, by the definition of $V_\mathrm{pr}$, the intersection $\overline{X}\cap \Sigma(V)$ is a union of $r$ adjacent edges $[x_0,x_1], [x_1,x_2], \ldots, [x_{r-1},x_r]$ of the skeleton $\Sigma(V)$ associated with $V$, for some $r\geqslant0$.
	Two cases are possible: either $x_r$ is a point of $V_\pr$, in which case $\partial X=\{x_0,x_r\}$ and $\overline{X}\cap \Sigma(V)=\overline{X}\cap \Sigma(V_\mathrm{pr})$ is an edge $[x_0,x_r]$ of $\Sigma(V_\mathrm{pr})$, or $\partial X=\overline{X}\cap \Sigma(V_\mathrm{pr})=\{x_0\}$.
	Let $\C_V$ be the snc model of $C$ associated with the snc vertex set $V$ and denote by $E_i$ the irreducible component of the special fiber $(\C_V)_k$ of $\C_V$ associated with $x_i$ and by $m_i$ the multiplicity of $E_i$ in $(\C_V)_k$.
	Then, for every $j=1,\ldots,r-1$, a standard intersection theoretic computation yields
	\[
	0 = \C_k \cdot E_j = \sum_{0\leqslant i\leqslant r} m_iE_i\cdot E_j
	=
	m_{j-1}+m_jE_j^2+m_{j+1}.
	\]
	We deduce that for every $j=0,\ldots,r-1$ we have
	\[
	\gcd\{m_j,m_{j+1}\}=\gcd\{m_0,\ldots,m_r\}.
	\]
	Moreover, if we are in the second case above, that is if $x_r$ is not a principal vertex of $V$, then we have $0 = (\C_V)_k \cdot E_r = m_{r-1}+m_rE_r^2$, and therefore $\gcd\{m_0,\ldots,m_r\}=m_r$.
}

\begin{theorem}\label{theorem_min-snc-triangulation}
	Assume that 
	the minimal extension $L$ of $K$ such that $C$ acquires semi-stable reduction over $L$ is tamely ramified.
	Then the following hold:
	\begin{enumerate}
		\item \label{theorem_min-snc-triangulation_one} the minimal snc vertex set $V_\mathrm{min-snc}$ of $C$ is a strong triangulation of $C$;
		\item \label{theorem_min-snc-triangulation_two} the principalization $(V_\mathrm{min-snc})_\mathrm{pr}$ of $V_\mathrm{min-snc}$ is the minimal strong triangulation of $C$.
	\end{enumerate}
\end{theorem}

\begin{proof}
	By Theorem~\ref{theorem_main_effective} and \ref{point:lcm_equal_in_tame_case} we know that $p$ does not divide the multiplicity of any element of the minimal strong triangulation $V_\mathrm{min-str}$ of $C$. 
	It then follows from Lemma~\ref{lemma_curves_dreaming_of_having_positive_genus} that $V_\mathrm{min-str}$ is contained in $V_\mathrm{min-snc}$.
	Now let $X$ be a connected component of $C\setminus V_\mathrm{min-str}$.
	Observe that the $\fraks(X)$-analytic space $X_{\fraks}$ defined in \ref{definition_space_over_field_of_constants} is a form of a disc or of an annulus with two ends trivialized by the tame extension $L|\fraks(X)$; it follows from \cite[Theorem 8.1]{FantiniTurchetti2018} that $X_{\fraks}$ is a fractional disc or a fractional annulus.
	It follows from the discussion of \ref{point_explicit_resolution_fractional_discs_and_annuli} that no principal point of the vertex set of the minimal snc model of $X_\s$ is contained in $X_\s$.
	This shows that $(V_\mathrm{min-snc})_\mathrm{pr}$ is contained in $V_\mathrm{min-str}$.
Now, in order to prove $(i)$, let $X$ be a connected component of $C\setminus V_\mathrm{min-snc}$.
Since $V_\mathrm{min-snc}$ is associated with a snc model $\C$ of $C$, the boundary $\partial X$ of $X$ has either one or two points.
In the first case, since $X$ does not contain any point of $V_\mathrm{min-str}$, it is contained in some virtual disc and therefore satisfies the hypotheses of Proposition~\ref{proposition_Ducros-criterion}, so that $X$ is itself a virtual disc.
In the second case, let $m_1$ and $m_2$ be the multiplicities of the two points of $\partial X$.
By \ref{point_intersection_multiplicities_chains} the greatest common divisor $\gcd\{m_1,m_2\}$ of $m_1$ and $m_2$ is not divisible by $p$, since otherwise we would find a point of $(V_\mathrm{min-snc})_\mathrm{pr}$ whose multiplicity is a multiple of $p$, which is not possible as we have shown that $(V_\mathrm{min-snc})_\mathrm{pr}\subset V_\mathrm{min-str}$.
Then $X$ is a generalized fractional annulus by Proposition~\ref{proposition_formalfibers}.
This proves that $V_\mathrm{min-snc}$ is a strong triangulation, showing $(i)$.
In order to prove $(ii)$, it remains to show that $(V_\mathrm{min-snc})_\mathrm{pr}$ is a strong triangulation of $C$.
Let $X$ be a connected component of $C\setminus(V_\mathrm{min-snc})_\mathrm{pr}$ and let $x_0\in (V_\mathrm{min-snc})_\mathrm{pr}$ be a boundary point of $X$.
Following the notation of \ref{point_intersection_multiplicities_chains} we write $\overline X\cap V_\mathrm{min-snc}=\{x_0,\ldots,x_r\}$.
Without loss of generality we can assume that $r\geqslant1$, because otherwise $U$ is a connected component of $C\setminus V_\mathrm{min-snc}$.
For every $i=1,\ldots,r$, denote by $X_i$ the connected component of $C\setminus V_\mathrm{min-snc}$ such that $\partial X_i=\{x_{i-1},x_{i}\}$.
Then, again by \ref{point_intersection_multiplicities_chains}, we have that $\gcd\{m_{i-1},m_{i}\}$ equals $\gcd\{m_0,\ldots,m_r\}$ and is not divisible by $p$, hence all the generalized fractional annuli $X_i$ have field of constants $\fraks(X_i)$ isomorphic to the unique totally ramified extension of $K$ of degree $\gcd\{m_0,\ldots,m_r\}$.
Moreover, we also have $\s(x_{i})=\s(X_i)$.
This follows from the fact that $\s(x_i)$ is contained in $\s(X_i)$ and $\s(X_i)$ is contained in the field of constants of each formal fiber at a closed point of the component of $\C_k$ associated with $x_i$.
Indeed, two of these components are $X_i$ and $X_{i+1}$, while all the others are associated with a special $R$-algebra of the form $R[[S,T]]/(uS^{m_{i}}-\pi)$ (again by the simple deformation-theoretic argument of \cite[Lemma 2.3.2]{ConradEdixhovenStein2003}), for some unit $u$ of $R[[S,T]]$, so that they contain the constant $(uS^{m_i})^{1/\gcd\{m_i,m_{i+1}\}}$.
We can then apply $r$ times the seond Fusion Lemma~\ref{lemma_fusion_annulus} to deduce that the connected component $X'$ of $C\setminus\{x_0,x_r\}$ containing $X_1$ is a virtual annulus with two ends. 
If $X'=X$ then there is nothing left to prove.
If $X'$ is a proper subset of $X$, this means that we are in the second case treated in \ref{point_intersection_multiplicities_chains} and thus $p$ does not divide $m_r$.
Now, if we blowup once a smooth point of $E_r$ and we denote by $x_{r+1}$ the type 2 point in $X$ corresponding to the exceptional component of this blowup, we have that $m_{r+1}=m_r$ and as before the connected component $X''$ of $C\setminus\{x_0,x_{r+1}\}$ containing $X_0$ is a virtual annulus with two ends.
The connected component of $C\setminus\{x_r\}$ containing $x_{r+1}$ being a virtual disc, by the first Fusion Lemma \ref{lemma_fusion_disc} the component $X$ is itself a virtual disc.
This proves that $(V_\mathrm{min-snc})_\mathrm{pr}$ is a strong triangulation, which concludes the proof of $(ii)$.
\end{proof}

Observe that the tameness hypothesis on $L|K$ is necessary. 
An example where the theorem does not hold in the wildly ramified case will be given in \ref{point_elliptic-curves-wild-potentially-multiplicative}.

The following effective version of the criterion by Saito cited in the introduction, which was originally due to Halle \cite[Theorem 7.5]{Halle2010}, follows immediately from the combination of Theorems~\ref{theorem_main_effective} and \ref{theorem_min-snc-triangulation}.

\begin{corollary}\label{corollary_saito_effective}
	Assume that 
	the minimal extension $L$ of $K$ such that $C$ acquires semi-stable reduction is tamely ramified.
	Then we have
	\[
	[L:K]=\lcm\big\{m(x)\,\big|\,x\in (V_\mathrm{min-snc})_\mathrm{pr}\big\}.
	\]
\end{corollary}

\begin{remark}
	In the classical formulation of Saito's criterion (cf. \cite[Theorem 3]{Saito1987}), the components of the special fiber $(\C_\mathrm{min-snc})_k$ of the minimal snc model $\C_\mathrm{min-snc}$ of $\frakC$ are required to satisfy the following additional conditions: no two components whose multiplicities are divisible by $p$ can intersect, and no component intersecting exactly only one other component can have multiplicity divisible by $p$.
	However, as follows from \ref{point_intersection_multiplicities_chains}, if $p$ divides the multiplicity of the last curve in a chain of rational curves in $(\C_\mathrm{min-snc})_k$, or if it divides the multiplicities of two intersecting components in such a chain, then it divides also the multiplicity of a principal component at the end of the chain.
	Therefore, the condition is automatically satisfied as long as no point of $(V_\mathrm{min-snc})_\mathrm{pr}$ has multiplicity divisible by $p$.
\end{remark}

\begin{remark}\label{remark_esssential-skeleton-tame}
	Let $V_{\mathrm{min-snc}}$ be the minimal snc vertex set of $C$.
	It is proven in \cite[Theorem~3.3.13]{BakerNicaise2016} that the subspace $\Sigma\big((V_{\mathrm{min-snc}})_\pr\big)$ of $C$ is the \emph{essential skeleton} of $C$, in the sense of \cite{MustataNicaise2015} and \cite{NicaiseXu2016}.
	If we assume that $C$ acquires semi-stable reduction over a tamely ramified extension of $K$, it follows then from Theorem~\ref{theorem_min-snc-triangulation_two}, that the essential skeleton coincides with $\Sigma(V_{\mathrm{min-str}})$.
	Moreover, the model associated with $V_{\mathrm{min-str}}$ is the minimal dlt model of $\frakC$ that is also minimal with respect to the domination order relation (see \cite[Remark~3.3.7]{BakerNicaise2016}).
	However, since Theorem~\ref{theorem_min-snc-triangulation_two} does not necessarily hold in the wildly ramified case, the essential skeleton and $\Sigma(V_{\mathrm{min-str}})$ do not always coincide.
	An example where the former is strictly smaller than the latter will be given in \ref{point_elliptic-curves-wild-potentially-multiplicative}.
\end{remark}


\section{Marked curves and tame triangulations of elliptic curves}
\label{section_elliptic_curves}

In this section we briefly explain how to adapt the theory that we have developed so far to the case of curves with marked points.
We then describe explicitly all possible minimal triangulations of elliptic curves, in terms of their reduction type, when the minimal extension yielding semi-stabiliy is tamely ramified.

\pa{
	Let $\frakC$ be a geometrically connected, smooth and projective algebraic curve over $K$ and let $\calN$ be a finite set of $K$-rational points of $\frakC$.
	We denote as usual by $C$ the analytification of $\frakC$, and we implicitly identify the set $\calN$ with the corresponding set of $K$-points of $C$.
	We then define a triangulation (respectively a strong triangulation) on $(C,\calN)$ as a triangulation (respectively a strong triangulation) $V$ of $C$ satisfying the following additional conditions: 
	\begin{enumerate}
		\item every connected component of $C \setminus V$ contains at most a point of $\calN$;
		\item every connected component of $C \setminus V$ containing a point of $\calN$ is a virtual disc.
	\end{enumerate}
	This definition ensures the existence of a minimal triangulation not only for curves of genus $g>1$, but more generally for curves of genus $g$ with $n$ marked points satisfying the numerical condition $2g-2+n>0$, and in particular for elliptic curves.
	Upon modifying the definitions of snc models and semi-stable models to require that all marked points specialize to smooth points of the special fiber, the definition of principal point of a snc vertex set by having each marked point contribute to the degree in \ref{point_def_principal} (so that, for example, a point $x$ of a snc vertex set $V$ of $E$ is principal also whenever $x$ has degree 2 in $\Sigma(V)$ and a connected component of $E\setminus V$ adjacent to $x$ contains a marked point), and redefining the analytic skeleton of $(C,\calN)$ as the analytic skeleton of $C\setminus\calN$ (that is the subset consisting of those points having no neighborhood in $C$ isomorphic to a virtual disc containing no marked point), then all results of the previous sections hold true more generally for marked curves.
	We leave to the reader to verify these results, observing that this task is made simpler by the fact that, if $\calC$ is a snc model of $C$, then every $K$-point of $C$ specializes to a smooth point of a component of multiplicity one of $\calC_k$ (see for example Remark~\ref{remark_regular_with_rational_point_is_disk}).
}

\Pa{Tame elliptic curves with potentially good reduction}{\label{point_elliptic_tame_additive}
	Let $(E,0)$ be the analytification of an elliptic curve over $K$ and assume that $E$ has potentially good reduction.
	Then the minimal triangulation of $(E,0)$ consists of a single point, the unique point $x$ of $E$ of genus $g(x)=1$.
	This is also the minimal strong triangulation.
	Assume that $E$ acquires semi-stable reduction after a tamely ramified extension $L$ of $K$ (which is always the case if the residue characteristic of $K$ is different from $2$ and $3$, since our $E$ has good reduction as soon as it admits a Weierstrass equation in Legendre normal form, which requires at most a base field extension of degree 6, see the proofs of VII.5.5 and III.1.7 in \cite{Silverman2016}).
	We are going to attach some natural numerical invariants to this triangulation, using the well known notion of the reduction type of $E$, that is of the combinatorics of the special fiber $\calE_k$ of the minimal snc model $\calE$ of $E$.
	The possible reduction types of elliptic curves can be classified and explicitly computed using Tate's algorithm (see for example \cite[IV.\S9]{Silverman1994}).
	As $E$ acquires good reduction after a tamely ramified extension of $K$, its reduction type is $\mathrm{I}_0$, $\mathrm{I}_0^\star$, $\mathrm{II}$, $\mathrm{II}^\star$, $\mathrm{III}$, $\mathrm{III}^\star$, $\mathrm{IV}$, or $\mathrm{IV^\star}$.
	We recall what this means in Figure~\ref{figure_potentially_good} below, where for every reduction type we depict the dual graph associated with $\calE_k$, 
	weighted on the vertices by the multiplicities of the corresponding irreducible component of $\calE_k$.
	In all cases, the vertex in red corresponds to the only principal component of $\calE_k$, 
	while we add an arrow departing from the vertex corresponding to the component onto which the marked point $0$ specializes.
	
	{
		\begin{figure}[h]
			\begin{center}
				\begin{tikzpicture}
				
				\draw[thin,>-stealth,->](0,0)--+(0.2,0.6);
				
				\draw[fill,color=red] (0,0)circle(2pt);
				
				\begin{footnotesize}
				\node(a)at(-0.1,.25){$1$};
				\end{footnotesize}
				
				\node(a)at(0,-1.3){$\mathrm{I}_0$};

				\begin{scope}[xshift=2cm]
				
				\draw[thin,>-stealth,->](.67,-0.7)--+(.2,0.6);		
				
				\draw[thin] (0,0)--(0.67,0.7);
				\draw[thin] (0,0)--(-.67,0.7);
				\draw[thin] (0,0)--(-.67,-0.7);
				\draw[thin] (0,0)--(.67,-0.7);
				\draw[fill,color=red] (0,0)circle(2pt);
				\draw[fill] (.67,0.7)circle(2pt);
				\draw[fill] (.67,-0.7)circle(2pt);
				\draw[fill] (-.67,0.7)circle(2pt);
				\draw[fill] (-.67,-0.7)circle(2pt);

				\begin{footnotesize}
				\node(a)at(0,.25){$2$};
				\node(a)at(-0.77,-.45){$1$};
				\node(a)at(-0.77,.45){$1$};
				\node(a)at(.77,.45){$1$};
				\node(a)at(.87,-.55){$1$};
				\end{footnotesize}

				\node(a)at(0,-1.3){$\mathrm{I}_0^\star$};
				
				\end{scope}

				\begin{scope}[xshift=4.5cm]
				
				\draw[thin,>-stealth,->](1,0)--+(.2,0.6);		
				
				\draw[thin] (0,0)--(1,0);
				\draw[thin] (0,0)--(-.65,0.7);
				\draw[thin] (0,0)--(-.67,-0.7);
				\draw[fill,color=red] (0,0)circle(2pt);
				\draw[fill] (1,0)circle(2pt);
				\draw[fill] (-.67,0.7)circle(2pt);
				\draw[fill] (-.67,-0.7)circle(2pt);

				\begin{footnotesize}
				\node(a)at(0.1,.25){$6$};
				\node(a)at(-0.77,-.45){$2$};
				\node(a)at(-0.77,.45){$3$};
				\node(a)at(.9,.25){$1$};
				\end{footnotesize}

				\node(a)at(0.2,-1.3){$\mathrm{II}$};
				
				\end{scope}

				\begin{scope}[xshift=6.5cm]
				
				\draw[thin,>-stealth,->](4.67,-0.7)--+(.2,0.6);
				
				\draw[thin] (0,0)--(1,0);
				\draw[thin] (0,0)--(.65,0.7);
				\draw[thin] (.65,0.7)--(1.65,0.7);
				\draw[thin] (0,0)--(.67,-0.7);
				\draw[thin] (.65,-0.7)--(4.65,-0.7);
				\draw[fill,color=red] (0,0)circle(2pt);
				\draw[fill] (1,0)circle(2pt);
				\draw[fill] (.67,0.7)circle(2pt);
				\draw[fill] (1.67,0.7)circle(2pt);
				\draw[fill] (.67,-0.7)circle(2pt);
				\draw[fill] (1.67,-0.7)circle(2pt);
				\draw[fill] (2.67,-0.7)circle(2pt);
				\draw[fill] (3.67,-0.7)circle(2pt);
				\draw[fill] (4.67,-0.7)circle(2pt);

				\begin{footnotesize}
				\node(a)at(-0.1,.25){$6$};
				\node(a)at(1.1,.25){$3$};
				\node(a)at(0.77,-.45){$5$};
				\node(a)at(1.77,-.45){$4$};
				\node(a)at(2.77,-.45){$3$};
				\node(a)at(3.77,-.45){$2$};
				\node(a)at(4.57,-.45){$1$};
				\node(a)at(0.77,.95){$4$};
				\node(a)at(1.77,.95){$2$};
				\end{footnotesize}

				\node(a)at(2.2,-1.3){$\mathrm{II}^\star$};

				\end{scope}

				\begin{scope}[yshift=-3.5cm]
				
				\draw[thin,>-stealth,->](1,0)--+(.2,0.6);		
				
				\draw[thin] (0,0)--(1,0);
				\draw[thin] (0,0)--(-.65,0.7);
				\draw[thin] (0,0)--(-.67,-0.7);
				\draw[fill,color=red] (0,0)circle(2pt);
				\draw[fill] (1,0)circle(2pt);
				\draw[fill] (-.67,0.7)circle(2pt);
				\draw[fill] (-.67,-0.7)circle(2pt);

				\begin{footnotesize}
				\node(a)at(0.1,.25){$4$};
				\node(a)at(-0.77,-.45){$1$};
				\node(a)at(-0.77,.45){$2$};
				\node(a)at(.9,.25){$1$};
				\end{footnotesize}

				\node(a)at(0.2,-1.3){$\mathrm{III}$};
				
				\end{scope}

				\begin{scope}[xshift=2.cm, yshift=-3.5cm]
				
				\draw[thin,>-stealth,->](2.67,-0.7)--+(.2,0.6);
				
				\draw[thin] (0,0)--(1,0);
				\draw[thin] (0,0)--(.65,0.7);
				\draw[thin] (.65,0.7)--(2.65,0.7);
				\draw[thin] (0,0)--(.67,-0.7);
				\draw[thin] (.65,-0.7)--(2.65,-0.7);
				\draw[fill,color=red] (0,0)circle(2pt);
				\draw[fill] (1,0)circle(2pt);
				\draw[fill] (.67,0.7)circle(2pt);
				\draw[fill] (1.67,0.7)circle(2pt);
				\draw[fill] (2.67,0.7)circle(2pt);
				\draw[fill] (.67,-0.7)circle(2pt);
				\draw[fill] (1.67,-0.7)circle(2pt);
				\draw[fill] (2.67,-0.7)circle(2pt);

				\begin{footnotesize}
				\node(a)at(-0.1,.25){$4$};
				\node(a)at(1.1,.25){$2$};
				\node(a)at(0.77,-.45){$3$};
				\node(a)at(1.77,-.45){$2$};
				\node(a)at(2.57,-.45){$1$};
				\node(a)at(0.77,.95){$3$};
				\node(a)at(1.77,.95){$2$};
				\node(a)at(2.77,.95){$1$};
				\end{footnotesize}

				\node(a)at(1.8,-1.3){$\mathrm{III}^\star$};

				\end{scope}

				\begin{scope}[yshift=-3.5cm,xshift=6.5cm]
				
				\draw[thin,>-stealth,->](1,0)--+(.2,0.6);		
				
				\draw[thin] (0,0)--(1,0);
				\draw[thin] (0,0)--(-.65,0.7);
				\draw[thin] (0,0)--(-.67,-0.7);
				\draw[fill,color=red] (0,0)circle(2pt);
				\draw[fill] (1,0)circle(2pt);
				\draw[fill] (-.67,0.7)circle(2pt);
				\draw[fill] (-.67,-0.7)circle(2pt);

				\begin{footnotesize}
				\node(a)at(0.1,.25){$3$};
				\node(a)at(-0.77,-.45){$1$};
				\node(a)at(-0.77,.45){$1$};
				\node(a)at(.9,.25){$1$};
				\end{footnotesize}

				\node(a)at(0.2,-1.3){$\mathrm{IV}$};
				
				\end{scope}

				\begin{scope}[xshift=8.5cm, yshift=-3.5cm]
				
				\draw[thin,>-stealth,->](2,0)--+(.2,0.6);		
				
				\draw[thin] (0,0)--(2,0);
				\draw[thin] (0,0)--(.65,0.7);
				\draw[thin] (.65,0.7)--(1.65,0.7);
				\draw[thin] (0,0)--(.67,-0.7);
				\draw[thin] (.65,-0.7)--(1.65,-0.7);
				\draw[fill,color=red] (0,0)circle(2pt);
				\draw[fill] (1,0)circle(2pt);
				\draw[fill] (2,0)circle(2pt);
				\draw[fill] (.67,0.7)circle(2pt);
				\draw[fill] (1.67,0.7)circle(2pt);
				\draw[fill] (.67,-0.7)circle(2pt);
				\draw[fill] (1.67,-0.7)circle(2pt);

				\begin{footnotesize}
				\node(a)at(-0.1,.25){$3$};
				\node(a)at(1.1,.25){$2$};
				\node(a)at(1.9,.25){$1$};
				\node(a)at(0.77,-.45){$2$};
				\node(a)at(1.77,-.45){$1$};
				\node(a)at(0.77,.95){$2$};
				\node(a)at(1.77,.95){$1$};
				\end{footnotesize}

				\node(a)at(1.2,-1.3){$\mathrm{IV}^\star$};

				\end{scope}

				\end{tikzpicture} 
			\end{center}
			\caption{}\label{figure_potentially_good}\end{figure}
	}
	
	\noindent As in Theorem~\ref{theorem_min-snc-triangulation}, the red vertex $x$ is also the unique point $x$ of the minimal triangulation of $(E,0)$, which in these cases is also the minimal strong triangulation of $(E,0)$.
	Thanks to Proposition~\ref{proposition_formalfibers} and to the tameness assumption, we can apply recursively the Fusion Lemmas 
	in order to show that the connected component $X$ of $E\setminus\{x\}$ that contains $0$ is a fractional disc, that is
	\[
	X\cong\big\{x\in\aan \,\big|\,|T(x)|<|\pi|^\alpha\big\},
	\]
	where $T$ denotes a coordinate function on $\aan$, for some rational number $\alpha$.
	Observe that the class of $\alpha$ in $\Q/\Z$, which depends only on the isomorphism class of $X$ and determines this isomorphism class uniquely, can be computed explicitly for every reduction type.
	Indeed, from Proposition~\ref{proposition_formalfibers} we obtain that the components of $E\setminus V_\mathrm{min-snc}$ with two ends are regular fractional annuli (in the terminology of Section~\ref{section_regularity_line}).
	We can then easily compute $\alpha$ by gluing together the regular fractional annuli of $E \setminus V_\mathrm{min-snc}$ in the path from $x$ to $0$.
	For example, in the case $\mathrm{IV^\star}$ we have to merge a disc centered at $0$ with a regular fractional annulus of the form 
	\(
	\big\{x\in\aan \;\big|\; |\pi| < |T(x)| < |\pi|^{1/2}\big\}
	\) 
	and with a second one of the form
	\(
	\big\{x\in\aan \;\big|\; |\pi|^{1/2} < |T(x)| < |\pi|^{1/3}\big\}.
	\)
	The resulting fractional disc $X$ will then be of the form
	\[
	X\cong\big\{x\in\aan \,\big|\,|T(x)|<|\pi|^{\sfrac{1}{3}}\big\},
	\]
	which is to say that $\alpha=1/3$.
	On the other hand, in the case $\mathrm{IV}$ we have to merge a disc centered at $0$ with a regular fractional annulus of the form 
		\(
		\big\{x\in\aan \;\big|\; |\pi|^{1/3} < |T(x)| < 1\big\}
		\cong 
		\big\{x\in\aan \;\big|\; |\pi| < |T(x)| < |\pi|^{1-1/3}\big\}
		\) 
		(observe that this is not the same as the fractional annulus $\big\{x\in\aan \;\big|\; |\pi| < |T(x)| < |\pi|^{1/3}|\big\}$, which is not regular)
		and so the resulting fractional disc $X$ will be of the form
		\[
		X\cong\big\{x\in\aan \,\big|\,|T(x)|<|\pi|^{1-\sfrac{1}{3}}\big\},
		\]
		that is we can take $\alpha=-1/3$.
	The other cases are analogue to the previous two.
	We deduce that the reduction type of $E$ is fully determined by (and determines uniquely) the datum of $\alpha$ and of the multiplicity $m(x)$ of $x$, as indicated in the following table:
	
	\smallskip
	{	\begin{center}
			\begin{tabular}{| l || c | c | c | c | c | c | c | c |}
				\hline
				{\bf reduction type} & $\mathrm{I}_0$ & $\mathrm{I}_0^\star$ & $\mathrm{II}$ & $\mathrm{II}^\star$ & $\mathrm{III}$ & $\mathrm{III}^\star$ & $\mathrm{IV}$ & $\mathrm{IV^\star}$\\ \hline
				\boldmath{$m(x)$} & 1 & 2 & 6 & 6 & 4 & 4 & 3 & 3 \\ \hline
				\boldmath{$\alpha \mod \Z$} & 0 & 1/2 & -1/6 & 1/6 & -1/4 & 1/4 & -1/3 & 1/3 \\
				\hline
			\end{tabular}
	\end{center}}
	\smallskip
	
	\noindent Note that in particular, as in Corollary~\ref{corollary_saito_effective}, the multiplicity $m(x)$ coincides with the degree $[L\colon K]$ of the minimal extension $L$ of $K$ over which $E$ acquires semi-stable reduction (that is, over whose valuation ring $E$ admits a smooth model).
}

\Pa{Tame elliptic curves with potentially multiplicative reduction}{\label{point_elliptic_tame_multiplicative}
	Let $E$ be the analytification of an elliptic curve which does not have potentially good semi-stable reduction.
	Equivalently, $E$ contains no type 2 point of positive genus. 
	Then $E$ acquires split multiplicative reduction over an extension $L$ of $K$, which means that $E_L$ has a snc model whose special fiber is a chain of rational curves closing onto itself.
	It follows that $E_L$ has the homotopy type of a circle and therefore $E$ is not geometrically contractible.
	Suppose that the residue characteristic of $K$ is different from two.
	Tate showed that the degree $[L:K]$ is at most two (see \cite[Theorem C.14.1.(d)]{Silverman2016}), so in our situation $L$ is tamely ramified over $K$.
	Over $K$, the curve $E$ can have reduction type either $\mathrm{I}_n$ (which means that $E$ has multiplicative reduction already over $K$) or $\mathrm{I}_n^\star$, for some $n>1$.
	In both cases $n$ is equal to the opposite of the $\pi$-adic valuation of the $j$-invariant of $E$.
	As before, we refer to \cite[IV.\S9]{Silverman1994} for a detailed discussion of reduction types and limit ourselves to recalling what this means in terms of the combinatorial data arising from the minimal snc model $\calE$ of $E$.
	Exactly as in Figure~\ref{figure_potentially_good}, Figure \ref{figure_potentially_multiplicative} depicts for each reduction type the dual graph associated with $\calE_k$, weighted with the multiplicities of the corresponding components of $\calE_k$, with the principal points depicted as red vertices, and an arrow departing from the vertex corresponding to the component onto which the point $0$ specializes.
	
	\begin{figure}[h]
			\begin{center}
				\begin{tikzpicture}
				
				\draw[thin,>-stealth,->](0,0)--+(-.2,0.6);		
				
				\draw[thin] (0,0)--(.65,0.7);
				\draw[thin] (.65,0.7)--(1.65,0.7);
				\draw[thin,dashed] (1.65,0.7)--(3,0.7);
				\draw[thin] (3,0.7)--(4,0.7);
				\draw[thin] (4,0.7)--(4.67,0);
				
				\draw[thin] (0,0)--(.67,-0.7);
				\draw[thin] (.65,-0.7)--(1.65,-0.7);
				\draw[thin,dashed] (1.65,-0.7)--(3,-0.7);
				\draw[thin] (3,-0.7)--(4,-0.7);
				\draw[thin] (4,-0.7)--(4.67,0);
				
				\draw[fill,color=red] (0,0)circle(2pt);
				\draw[fill] (.67,0.7)circle(2pt);
				\draw[fill] (1.67,0.7)circle(2pt);
				\draw[fill] (3,0.7)circle(2pt);
				\draw[fill] (4,0.7)circle(2pt);
				\draw[fill] (.67,-0.7)circle(2pt);
				\draw[fill] (1.67,-0.7)circle(2pt);
				\draw[fill] (3,-0.7)circle(2pt);
				\draw[fill] (4,-0.7)circle(2pt);
				\draw[fill] (4.67,0)circle(2pt);

				\begin{footnotesize}
				\node(a)at(-0.1,-.25){$1$};
				
				\node(a)at(4.77,.25){$1$};
				
				\node(a)at(0.77,.95){$1$};
				\node(a)at(1.77,.95){$1$};
				\node(a)at(3.1,.95){$1$};
				\node(a)at(4.1,.95){$1$};
				
				\node(a)at(0.77,-.45){$1$};
				\node(a)at(1.77,-.45){$1$};
				\node(a)at(3.1,-.45){$1$};
				\node(a)at(3.9,-.45){$1$};
				
				\node(a)at(2.4,-1.2){(loop consisting of $n>1$ edges)};
				\end{footnotesize}
				
				\node(a)at(2.43,-1.9){$\mathrm{I}_n$};

				\begin{scope}[xshift=7cm]
				
				\draw[thin,>-stealth,->](-.67,-0.7)--+(-.2,0.6);		
				
				\draw[thin] (0,0)--(1,0);
				\draw[thin,dashed] (1,0)--(2.43,0);
				\draw[thin] (2.43,0)--(3.43,0);
				\draw[thin] (0,0)--(-.65,0.7);
				\draw[thin] (0,0)--(-.67,-0.7);
				\draw[thin] (3.43,0)--(4,0.7);
				\draw[thin] (3.43,0)--(4,-0.7);			
				
				\draw[fill,color=red] (0,0)circle(2pt);
				\draw[fill] (1,0)circle(2pt);
				\draw[fill] (2.43,0)circle(2pt);
				\draw[fill,color=red] (3.43,0)circle(2pt);
				\draw[fill] (-.67,0.7)circle(2pt);
				\draw[fill] (-.67,-0.7)circle(2pt);
				\draw[fill] (4,0.7)circle(2pt);
				\draw[fill] (4,-0.7)circle(2pt);

				\begin{footnotesize}
				\node(a)at(0.1,.25){$2$};
				\node(a)at(1.1,.25){$2$};
				\node(a)at(-0.57,-.95){$1$};
				\node(a)at(-0.57,.95){$1$};
				\node(a)at(2.53,.25){$2$};
				\node(a)at(3.33,.25){$2$};
				\node(a)at(4.1,.45){$1$};
				\node(a)at(4.1,-.95){$1$};
				
				\draw [decorate,decoration={brace,amplitude=5pt,mirror,raise=3ex}]
				(0,0) -- (3.43,0) node[midway,yshift=-3em]{($n>1$ edges)};
				
				\end{footnotesize}

				\node(a)at(1.8,-1.9){$\mathrm{I}^\star_n$};
				\end{scope}
				\end{tikzpicture} 
			\end{center}
			\caption{}\label{figure_potentially_multiplicative}
	\end{figure}

	\noindent As in Theorem \ref{theorem_min-snc-triangulation}, the red points correspond to the minimal strong triangulation, but they do not correspond to the minimal triangulation in the case of $\mathrm{I}_n^\star$. 
	Indeed, in this case the minimal triangulation contains only the red point on the left, since removing this latter cuts the curve into a virtual annulus with one end and an infinite number of generalized fractional discs.
	As an aside, we observe that keeping only the red point on the right would yield a triangulation of the unmarked curve $E$ which is not a triangulation of $(E,\{0\})$, since the origin would lie in the virtual annulus; in particular the unmarked curve $E$ does not have a minimal triangulation.
	For both types $\mathrm{I}_n$ and $\mathrm{I}_n^\star$, by reasoning as in the case of potentially good reduction, we can compute the isomorphism class of the unique virtual annulus $X$ among the connected components of $E \setminus V_{\mathrm{min-str}}$.
	In the case $\mathrm{I}_n^\star$ this is a generalized fractional annulus trivialized by a quadratic extension $L$ of $K$, and it is then determined by the multiplicities of its ends and the isomorphism class of $X_L$.
	We then obtain the following table, associating with the datum of the minimal strong triangulation and the isomorphism class of the virtual annulus $X$ a uniquely determined reduction type.
	
	\smallskip
	{	\begin{center}
			\begin{tabular}{| l || c | c |}
				\hline
				{\bf reduction type} & $\mathrm{I}_n$ & $\mathrm{I}_n^\star$\\ \hline
				\boldmath{$\#V_{\mathrm{min-str}}$} & 1 & 2 \\ \hline
				\boldmath{$X\subset\aan$} & $\big\{|\pi|^n<|T(x)|<1\big\}$ & $\big\{|\pi|^{2(n+1)} < |T^2(x)-\pi| < |\pi|^{2}\big\}$\\
				\hline
			\end{tabular}
	\end{center}}
	\smallskip
	
	This completes the description of the minimal triangulations and strong triangulations of elliptic curves in the tame case.
}

\section{First steps in wildlife observation}

\label{section_wild}

In this section, we discuss several examples of curves that acquire semi-stable reduction after a wildly ramified extension of $K$. 
We explain why in some cases dropping the tameness assumption leads to a failure of the equality of the effective version of Saito's criterion (Corollary~\ref{corollary_saito_effective}), with the aim to start a systematic study of the minimal extension realizing semi-stable reduction from the point of view of non-archimedean analytic geometry.\\

\Pa{Elliptic curves with potentially multiplicative reduction}{\label{point_elliptic-curves-wild-potentially-multiplicative}
	Assume that $K$ has residue characteristic 2 and let $E$ be the analytification of an elliptic curve over $K$ with non-split multiplicative reduction, that is, $E$ has not multiplicative reduction over $K$ but it acquires it after a base change to a finite separable extension $L$ of $K$.
	In this case, such an extension $L$ can always be taken to be of degree 2 over $K$.
	The analytic skeleton of $E$ is a line segment, since it is the quotient of the loop $\Sigma^{\mathrm{an}}(E_L)$ by the involution induced by the action of the Galois group $\mathrm{Gal}(L|K)$.
	The minimal strong triangulation of $E$ consists precisely of the two endpoints $x_1$ and $x_2$ of this segment, which are the images via the base change morphism of the two points of $\Sigma^{\mathrm{an}}(E_L)$ that are fixed by the action of $\mathrm{Gal}(L|K)$.
	Equivalently, $x_1$ and $x_2$ are the only points of the analytic skeleton $\Sigma^{\mathrm{an}}(E)$ of $E$ whose fields of constants $\fraks(x_1)$ and $\fraks(x_2)$ are equal to $K$.
	These are the only nodes of the analytic skeleton of $E$, since any other point $y$ of $\Sigma^{\mathrm{an}}(E)$ satisfies $g(y)=0$ and $\fraks(y)=L$, therefore they form the minimal strong triangulation of $E$.
	The minimal triangulation of $E$ is obtained from the minimal strong triangulation by keeping only the point, say $x_1$, corresponding to the component onto which the neutral element of $E$ specializes.
	The connected component of $E\setminus\{x_1\}$ which contains $x_2$ is then a virtual annulus with one end and bending point $x_2$, trivialized by $L$.
	\\
	Let us discuss an explicit example.
	Let $K=\widehat{\Q_2^\mathrm{ur}}$ be the completion of the maximal unramified extension of the field of $2$-adic numbers and let $E$ be the analytification of the elliptic curve over $K$ defined birationally by the minimal Weierstrass equation
	\[
		v^2-2uv+16v=u^3+2u^2+32u.
	\]
	The $2$-adic valuation of the $j$-invariant $j(E)$ of $\frakE$ is $v_K(j(E))=-1$, and hence the curve $\frakE$ has potentially multiplicative reduction.
	It does not have multiplicative reduction itself, as the $R$-model defined by the same equation has additive reduction.
	To be more precise, Tate's algorithm shows that the reduction type of $\frakE$ is $I_5^\star$.
	However, $\frakE$ has multiplicative reduction over a totally ramified degree 2 extension $L$ of $K$, we have $v_L(j(E))=-2$ and therefore $\frakE_L$ has reduction type $I_2$.
	The morphism of dual graphs induced by the base change to $L$ is depicted in Figure~\ref{figure_wild_multiplicative_reduction} below.
	\begin{figure}[h]
		\begin{tikzpicture}[scale=1]	
		
		\draw[thin,>-stealth,->](5.8,2.4)--+(.4,0.5);		
		\draw[thin,>-stealth,->](5.8,-0.6)--+(.4,0.5);		
		
		\draw[thin] (0,0)--(5,0);
		\draw[thin] (0,0)--(-.8,0.6);
		\draw[thin] (0,0)--(-.8,-0.6);
		\draw[thin] (5,0)--(5.8,0.6);
		\draw[thin] (5,0)--(5.8,-0.6);
		\draw[fill,color=darkblue] (0,0)circle(2pt);
		\draw[fill] (1,0)circle(2pt);
		\draw[fill,color=red] (2,0)circle(2pt);
		\draw[fill,color=red] (3,0)circle(2pt);
		\draw[fill] (4,0)circle(2pt);
		\draw[fill,color=darkblue] (5,0)circle(2pt);
		\draw[fill] (-.8,0.6)circle(2pt);
		\draw[fill] (-.8,-0.6)circle(2pt);
		\draw[fill] (5.8,0.6)circle(2pt);
		\draw[fill] (5.8,-0.6)circle(2pt);

		\draw[thin,>-stealth,->](2.5,2)--(2.5,0.5);

		
		\draw[thin,dashed] (2,3)--(0,3);
		\draw[thin,dashed] (0,3)--(-.8,3.6);
		\draw[thin,dashed] (0,3)--(-.8,2.4);
		
		\draw[thin,dashed] (3,3)--(5,3);
		\draw[thin,dashed] (5,3)--(5.8,3.6);
		\draw[thin,dashed] (5,3)--(5.8,2.4);
		
		\draw (2,3) to[bend right] (3,3);
		\draw (2,3) to[bend left] (3,3);
		
		\draw[fill=white] (0,3)circle(2pt);
		\draw[fill=white] (1,3)circle(2pt);
		\draw[fill,color=red] (2,3)circle(2pt);
		\draw[fill=white] (4,3)circle(2pt);
		\draw[fill,color=red] (3,3)circle(2pt);
		\draw[fill=white] (5,3)circle(2pt);

		\draw[fill=white] (-.8,3.6)circle(2pt);
		\draw[fill=white] (-.8,2.4)circle(2pt);
		\draw[fill=white] (5.8,3.6)circle(2pt);
		\draw[fill=white] (5.8,2.4)circle(2pt);

		
		\node(a)at(-2,0){$I^\star_5$};
		\node(a)at(-2,3){$I_{2}$};


		\end{tikzpicture} 
		\caption{}
		\label{figure_wild_multiplicative_reduction}
	\end{figure}

	In both graphs in the figure, the red vertices form the minimal strong triangulation, while the minimal triangulation consists of the red vertex on the right since that's the vertex carrying the arrow once all non-red vertices have been contracted.
	In the top graph, the two red vertices actually also form the minimal snc vertex set of $E_L$; the dashed part of the graph, which corresponds to two discs, contains the components that are contracted in order to pass to the minimal snc model of $E_L$.
	In the bottom graph, $(V_\mathrm{min-snc})_\mathrm{pr}$ consists of the two blue vertices, and is therefore different from $V_\mathrm{min-snc}$.
	In particular, the second part of Theorem~\ref{theorem_min-snc-triangulation} does not hold.
	In this case such a behavior can also be seen as a particular case of \cite[Theorem 2.8]{Lorenzini2010}; other examples of this phenomenon, including in equicharacteristic 2, can be found in the proof of that result.
}

\Pa{Elliptic curves with potentially good reduction}{
We now give examples concerning points of positive genus.
Let $E$ be the analytification of an elliptic curve with potentially good reduction.
Then the minimal triangulation of $E$ coincides with its minimal strong triangulation and it consists of a single point, the unique type 2 point $x$ of $E$ of positive genus.
Let us discuss some consequences of this simple observation.
\begin{enumerate}
	\item Let $E$ be the analytification of the elliptic curve $\frakE$ over $K$ defined birationally by the equation
	\[
		v^2-\pi^2v= u^3+\pi u^2 +\pi^3u.
	\]
	If $K$ has residue characteristic 2, a simple computation shows that the $j$-invariant of $\frakE$ has positive valuation, so that $\frakE$ has potentially good reduction.
	Applying Tate algorithm, one can show that the curve $E$ has reduction type $\mathrm{I}_1^\star$.
	As a result, the skeleton of the minimal snc vertex set of $E$ is as depicted in Figure~\ref{figure_I_1^8} below.
	\begin{figure}[h]
		\begin{tikzpicture}[scale=1]	
		
		\draw[thin,>-stealth,->](1.8,-0.6)--+(+.4,0.5);		
						
		\draw[thin] (0,0)--(1,0);
		\draw[thin] (0,0)--(-.8,0.6);
		\draw[thin] (0,0)--(-.8,-0.6);
		\draw[thin] (1,0)--(1.8,0.6);
		\draw[thin] (1,0)--(1.8,-0.6);
		\draw[fill,color=red] (0,0)circle(2pt);
		\draw[fill] (1,0)circle(2pt);
		\draw[fill,color=red] (1,0)circle(2pt);
		\draw[fill] (-.8,0.6)circle(2pt);
		\draw[fill] (-.8,-0.6)circle(2pt);
		\draw[fill] (1.8,0.6)circle(2pt);
		\draw[fill] (1.8,-0.6)circle(2pt);

		\node(a)at(-2,0){$\mathrm{I}^\star_1$};
		\end{tikzpicture} 
		\caption{}\label{figure_I_1^8}
	\end{figure}
	The two red vertices in the figure are exactly the principal points of $V_\mathrm{min-snc}$.
	In particular, since $V_\mathrm{min-str}$ consists of a single point of $E$, we have $(V_\mathrm{min-snc})_\mathrm{pr}\neq V_\mathrm{min-str}$, and therefore Theorem~\ref{theorem_min-snc-triangulation} implies that the minimal Galois extension $L$ of $K$ such that $E$ acquires semi-stable reduction over $L$ is wildly ramified.
	More generally, the same argument shows that an elliptic curve that has potentially good reduction and reduction type $\mathrm{I}_n^\star$, which can only exist in residue characteristic 2, acquires semi-stable reduction after a wild extension of its base field.
	
	\item \label{example:wild_bad_cusp}
	Assume that $K$ has mixed characteristic $(0,2)$ and let $E$ be the analytification of the elliptic curve over $K$ defined birationally by the equation
	\[
	v^2-u^3=\pi.
	\]
	Then $E$ has reduction type $\mathrm{II}$ (as depicted in Figure~\ref{figure_potentially_good} in the previous section).
	Indeed, the minimal snc model of $E$ can be computed explicitly quite simply, as this boils down to the classical computation of a good embedded resolution of the plane cuspidal curve defined by $v^2-u^3=0$.
	If we denote by $\varpi$ a square root of $\pi$, the change of variables defined by $w=u/\sqrt[3]{4\pi}$ and $z=(v-\varpi)/2\varpi$ brings the equation to
	\[
	w^3=z(z-1),
	\]
	which is smooth.
	Observe that, whenever the $\pi$-adic valuation of $2$ is congruent to $1$ modulo $3$, then $\sqrt[3]{4\pi}$ belongs to the quadratic extension $L=K(\varpi)$ of $K$.
	It follows that in this case $E$ acquires good semi-stable reduction after the base change to $L$.
	In particular, $x$ has multiplicity at most $2$, and therefore it cannot be the point of $(V_\mathrm{min-snc})_\mathrm{pr}$, as the latter has multiplicity 6.
	In fact, it is interesting to point out that, contrarily to what happened in Example~\ref{point_elliptic-curves-wild-potentially-multiplicative}, the point $x$ is not in $V_\mathrm{min-snc}$ either.
	Indeed, let $y$ be the only point of multiplicity 2 of $V_\mathrm{min-snc}$, which is the point associated with the exceptional component of the blowup of the origin of the special fiber of the $R$-model defined by $v^2-u^3=\pi$.
	Then $x$ does not coincide with $y$, as can be verified directly on the algebra of the blowup, for example by showing that the inverse image of $y$ via the base change map to $L$ consists of a point which still has multiplicity 2.
	The point $x$ can be obtained by blowing up further a point of the exceptional component associated with $y$ (more precisely, the only point of the component which becomes not regular in the normalized base change of the blown-up model to $R[\varpi]$).
	The number of blowups required to make $x$ appear depend on the $\pi$-adic valuation of 2; for example, one blowup is sufficient if this valuation is equal to 1.
	We remark that the same behavior occurs with the hyperelliptic curves defined by the equation $v^2-u^m=\pi$ (see also \cite[Remark~3.12]{Lorenzini2010}).
	
	\item 
	In the example above, the point of positive genus of $E$ is not contained in $V_\mathrm{min-snc}$.
	We give a condition that prevents this behavior.
	Let $L$ be a Galois extension of $K$ and let $\calE$ be a smooth model of $E_L$ over the valuation ring of $L$.
	Then $G=\Gal(L|K)$ acts on the special fiber $\calE_k$ of $\calE$, inducing a Galois cover of $k$-curves $\varphi\colon \calE_k \to \calE_k/G$.
	Assume that the ramification locus of $\varphi$ consist of at least three distinct points. 
	Then $x$ is a principal point of $V_\mathrm{min-snc}$. 
	Indeed, the ramification points of $\varphi$ are in one-to-one correspondence with the edges of the skeleton $\Sigma(V)$ that are adjacent to $x$, where $V$ is the minimal snc vertex set of $E$ that contains $x$, since those are precisely the point of $\calE_k$ that map to singular point of $\calE_k/G$ by \cite[\S5.2]{Lorenzini2014}. 
	If $x$ is not a point of $V_\mathrm{min-snc}$, then $V$ is obtained by adding $x$ to a snc vertex set of $E$, so that $x$ has degree one or two in $\Sigma(V)$.
	As this contradicts our hypothesis, we deduce that $x$ is a point of $V_\mathrm{min-snc}$, and hence of $(V_\mathrm{min-snc})_\mathrm{pr}$.	
\end{enumerate}
}

\Pa{Example}{
In general, one pathology that may arise in the wildly ramified case comes from the fact that desingularizing a virtual disc may result in the creation of new principal components, leading to points of $(V_\mathrm{min-snc})_\mathrm{pr}$ that are not in $V_\mathrm{min-tr}$.
This is the case of the two connected components of $E\setminus V_\mathrm{min-tr}$ that contain a trivalent vertex in Figure~\ref{figure_wild_multiplicative_reduction} in Example~\ref{point_elliptic-curves-wild-potentially-multiplicative}, and of the connected component of $E\setminus V_\mathrm{min-snc}$ containing the origin of the elliptic curve in Example~\ref{example:wild_bad_cusp}.
Another explicit example of such a virtual disc can be realized as a subspace of the $K$-analytic projective line $\pan$ as follows.
Consider the subspace $X$ of $\pan$ defined as
	\[
	X=\big\{x\in\pan \;\big|\; |(T^p-\pi)(x)|>|\pi|\big\}. 
	\]
Observe the similarity of this example with Example~\ref{example_disco_stronzo}.
Then $X$ is a virtual disc, as can be deduced from Proposition~\ref{proposition_Ducros-criterion} (it also possible to see explicitly that, after adding $p$-th root of $\pi$ to $K$, the space $X$ becomes isomorphic to a disc centered at infinity).
However, $X$ is not a generalized fractional disc, and the special fiber of the minimal snc desingularization of the canonical model of $X$ contains one component that intersects three other components.
%
}

\appendix

\section{Open fractional annuli and regularity}
\label{section_regularity_line}

In this section we introduce a notion of regularity for semi-affinoid spaces.
We then focus on the regular open semi-affinoid subspaces of $\aan$, giving an interpretation using continued fractions.

\pa{
	We say that a semi-affinoid $K$-analytic space $X$ is \emph{regular} if the associated special $R$-algebra $\calO^\circ(X)$ is regular.
}

\examples{\label{examples_regular_semi-affinoids}
	An annulus of the form 
	\[
A_{n,K}=\big\{x\in\mathbb A^{1,\mathrm{an}}_K\;\big|\;|\pi^n|<|T(x)|<1\big\}
	\]
	is regular if and only if $n=1$ is one.
	Indeed, we have $\calO^\circ(A_{n,K})\cong R[[S,T]]/(ST-\pi^n)$.
	A fractional disc of the form
	\[
	X=\big\{x\in\aan \;\big|\; |T(x)|<|\pi|^{\sfrac{1}{d}}\big\}
	\]
	is regular if and only if $d=1$, that is if and only if it is a disc (see Example~\ref{example_fractional_disc}).
	Other examples of regular semi-affinoid spaces are the virtual discs of Examples~\ref{example_disco_stronzo} and \ref{example:virtual_several_fractional_discs}.
}

\pa{
Let $X$ be an open semi-affinoid subspace of $\pan$ whose boundary $\partial X$ in $\pan$ consists of finitely many type 2 points, and let $\C$ be the normal model of $\pan$ whose vertex set is $\partial X$.
Recall that, as discussed in~\ref{point_formal_fibers} and~\ref{theorem_models-vertex-sets}, there exists a closed point $P$ of the special fiber of $\C$ such that $X\cong\Sp^{-1}_{\C}(P)$, and thus $\widehat{\calO_{\C,P}}\cong\calO^\circ(X)$.
In particular, $X$ is regular if and only if $\C$ is regular at the point $P$.
Equivalently, $X$ is regular if and only if $X\cap V=\emptyset$, where $V$ is the minimal vertex set of $\pan$ that contains $\partial X$ and whose associated model of $\pan$ is regular.
}

\pa{\label{point_pre_explicit_resolution_fractional_discs_and_annuli}
Assume that $X$ is contained in $\aan$.
Then one can check the regularity of $X$ as follows.
Since $X$ is bounded, there exists a point $x_0$ of $\pan\setminus X$ of type 2 that has multiplicity $m(x_0)=1$.
Then $\{x\}$ is the vertex set of a smooth model $\C_0$ of $\pan$, and there exists a minimal sequence of point blowups $\C_1\to \C_0$ such that $\partial X$ is contained in the snc vertex set $V_{\C_1}$ associated with $\C_1$.
Then $X$ is regular if and only if it contains no point of $V_{\C_1}$.
Indeed, this is a simple consequence of the fact that $V_{\C_1}$ contains the minimal vertex set $V_{\C_2}$ that contains $\partial X$ and whose associated model $\C_2$ is regular, that the resulting morphism $\C_1\to \C_2$, being a morphism of regular models, is a sequence of point blowups, and of the minimality in the definition of $\C_1$.
We will use this observation to study the regularity of fractional annuli in Lemma~\ref{proposition_regular-fractional-annuli}.
}

\begin{example}
\label{example_fractional_annulus_regularization_blowup}
Let us performe the procedure of the previous point in two simple cases.
Consider the two fractional annuli 
\[ 
X=\big\{x\in\aan \;\big|\; |\pi|^{{1}/{2}} < |T(x)| < 1\big\}
\]
and
\[
X'=\big\{x\in\aan \;\big|\; |\pi|^{{2}/{3}} < |T(x)| < 1\big\}.
\]
Observe that the vertex set of the model $\C_0=\PP^{1,\mathrm{an}}_R$ of $\pan$ consist of a single point $x_0$, the boundary of the unit disc centered at $0$, which belongs to the boundary of both $X$ and $X'$.
By blowing up $\C_0$ along the origin of its special fiber, we obtain a model $\C_1$ of $\pan$ whose special fiber consists of two rational curves cutting each other transversally in one point.
Its vertex set is $V_{\C_1}=\{x_0,x_1\}$, where $x_1$ is the boundary point of the disc of center 0 and radius $|\pi|$ in $\pan$.
If we blow up $\C_1$ along the double point of its special fiber we obtain a new model $\C_2$ of $\pan$ whose vertex set is $V_{\C_2}=\{x_0,x_1,x_2\}$, where $x_2$ is the boundary point of the disc of center 0 and radius $|\pi|^{1/2}$ in $\pan$.
In particular, $V_{\C_2}$ contains the boundary $\partial X=\{x_0,x_2\}$ of $X$ and $X\cap V_{\C_2}=\emptyset$, so that according to~\ref{point_pre_explicit_resolution_fractional_discs_and_annuli} the fractional annulus $X$ is regular.
Consider now the model $\C_3$ of $\pan$ obtained by blowing up $\C_2$ along the intersection point between the two components of its special fiber associated with $x_1$ and $x_2$.
Its vertex set is $V_{\C_3}=\{x_0,x_1,x_2,x_3\}$, where $x_3$ is the boundary point of the disc of center 0 and radius $|\pi|^{2/3}$ in $\pan$, and thus it contains the boundary $\partial X'=\{x_0,x_3\}$ of $X'$.
However, $X\cap V_{\C_3}=\{x_2\}\neq \emptyset$, and thus according to~\ref{point_pre_explicit_resolution_fractional_discs_and_annuli} the fractional annulus $X'$ is not regular.
This can also be checked directly by looking at the special $R$-algebras associated with $X$ and $X'$, which are described in~\cite[\S7]{FantiniTurchetti2018}.
\end{example}

\pa{\label{point_explicit_resolution_fractional_discs_and_annuli}
If $X$ is any fractional annulus, then we can perform the procedure of~\ref{point_pre_explicit_resolution_fractional_discs_and_annuli} like in the previous example, starting from the model $\PP^{1,\mathrm{an}}_R$, and at every step we either blow up along the smooth point of the special fiber where the origin of $\pan$ specializes, or we blow up a double point of the special fiber.
This yields a model of $\pan$ whose vertex set $V$ is contained in the line segment joining $0$ to $\infty$ in $\pan$.
In particular, $V\cap X$ is contained in the skeleton of $X$ and the connected components of $X\setminus V$ are regular virtual discs and finitely many regular fractional annuli.
For example, in~\ref{example_fractional_annulus_regularization_blowup} the fractional annulus $X'$ is cut by the type 2 point $x_2$ into two regular fractional annuli and a family of virtual discs.
In particular, in the terminology of~\ref{point_def_principal}, this implies that the vertex set of the minimal snc model of $X$ dominating its canonical model has no principal points.
Similarly, a fractional disc $X$ can be cut into regular virtual discs and finitely many regular fractional annuli by removing finitely many type 2 points that lie on the line segment between its boundary point and any chosen $K$-rational point of $X$, so that the same conclusion about principal points holds.
}

\begin{remark}	
	It follows from the discussion above that every regular open semi-affinoid subset of $\aan$ has at most two boundary points.
	For example, this shows that the semi-affinoid space of special $R$-algebra $R[[X,Y]]/(XY(X-Y)-\pi)$ cannot be embedded in the analytic line.
\end{remark}

\Pa{Convention}{
In the following lemma, given a rational number $\alpha$ we consider its Euclidean continued fraction expansion
\[
\alpha = a_0+\cfrac{1}{a_1+\cfrac{1}{\cdots+\cfrac{1}{a_n}}}
\]
that is uniquely determined by requiring that the integers $a_i$ satisfy the conditions $a_1,\ldots,a_{n-1}\geqslant1$ and $a_n>1$.
We denote this continued fraction expansion by $\alpha=[a_0;a_1,\ldots,a_n]$.
\\
Recall also that by convention $\gcd(a,0)=a$ for every positive integer $a$.
}

\begin{lemma}\label{proposition_regular-fractional-annuli}
		Let $X\subset \aan$ be an open fractional annulus of radii $|\pi|^{\sfrac{a}{b}}$ and $|\pi|^{\sfrac{a'}{b'}}$, with ${a},{b}, {a'},{b'}\geqslant 0$, $\gcd(a,b)=\gcd(a',b')=1$, and $b\leqslant b'$, so that $b$ and $b'$ are the multiplicities of the two ends of $X$.
	Then the following conditions are equivalent:
	\begin{enumerate}
		\item $X$ is regular;
		\item either $b'=b=1$ and $a'=a\pm1$, or $b'>b\geqslant1$, the continued fraction expansion of ${{a}/{b}}$ is  $[a_0;a_1,\ldots,a_{n-1},a_n]$, and the continued fraction expansion of ${{a'}/{b'}}$ is one of the following:
		\begin{itemize}
		\setlength\itemsep{.3em}
			\item $[a_0;\ldots,a_{n-1},a_n+1]$;
			\item $[a_0;\ldots,a_{n-1},a_n,a_{n+1}]$ for some integer $a_{n+1}\geqslant2$;
			\item $[a_0;\ldots,a_{n-1},a_n-1,2]$;
			\item $[a_0;\ldots,a_{n-1},a_n-1,1,a_{n+2}]$ for some integer $a_{n+2}\geqslant2$.
		\end{itemize}
		\vspace{.3em}
		\item $X$ contains no $K(\pi^{\sfrac{1}{d}})$-rational point, for every $d\leqslant\max\{b, b'\}$; 
		\item the skeleton of $X$ contains no point of multiplicity $d \leqslant\max\{b, b'\}$;
		\item the matrix 
		$\begin{pmatrix}
			a & a' \\
			b & b'
		\end{pmatrix}$
		has determinant $\pm1$.
	\end{enumerate}
\end{lemma}

\begin{proof}
	Let us prove the equivalence between $(i)$ and $(ii)$.
	If $b'=b$, then $X$ is regular if and only if $b'=b=1$ and $a'=a\pm1$, as can be readily checked by applying the procedure of~\ref{point_pre_explicit_resolution_fractional_discs_and_annuli}.
	Therefore we can assume that $b'>b$.
	Denote by $x$ (respectively $x'$) the type $2$ point of $\pan$ which is the boundary of the disc centered at the origin and of radius $|\pi|^{\sfrac{a}{b}}$ (respectively $|\pi|^{\sfrac{a'}{b'}}$).
	Then, as observed in \ref{point_pre_explicit_resolution_fractional_discs_and_annuli}, the virtual annulus $X$ is regular if and only if it is a formal fiber of the smallest snc model of $\pan$ whose vertex set contains both $x$ and $x'$.
%
	Observe that the minimal snc model $\X$ of $\pan$ that dominates $\PP^{1,\mathrm{an}}_R$ and whose vertex set contains $x$ is obtained by blowing up $\PP^{1,\mathrm{an}}_R$ $a_0+1$ times \emph{downwards} (that is, blowing up the point of the exceptional divisor of the previous blowup that points towards zero), then $a_1$ times \emph{upwards} (that is, blowing up the point of the exceptional divisor of the previous blowup that points towards $\infty$), then $a_2$ times downwards, $a_3$ upwards and so on up to $a_n-1$ times upwards (if $n$ is odd) or downwards (if $n$ is even).
	This follows from the fact that when we blow up the intersection point of two divisors whose associated type 2 points have radii $|\pi|^{\sfrac{a}{b}}$ and $|\pi|^{\sfrac{c}{d}}$, we obtain the type 2 point of radius $|\pi|^{\sfrac{a+c}{b+d}}$, and a standard argument based on the Euclidean algorithm (in particular, note that denominators never simplify).
	Then, $X$ is regular if and only if $x'$ is contained in the vertex set of the snc model $\X'$ obtained by further blowing up $\X$ exactly once at a closed point of the divisor corresponding to $x$, either upwards or downwards, and then finitely many times in the opposite direction, as any other further blowup would add some point of the path between $x$ and $x'$ to the corresponding vertex set.
	Then one concludes that the equivalence between $(i)$ and $(ii)$ holds by carefully tracking what happens with the continued fraction expansion of $a/b$ after performing upward or downward blowups.
	For example, the second expansion of the list corresponds to a model $\X'$ obtained by blowing up once in the same direction as the last blowup in $\X\to\PP^{1,\mathrm{an}}_R$ and then $a_{n+1}$ in the other direction; the other cases are similar.
	The equivalence with property $(v)$ follows now from the standard criterion for regularity in toric geometry, which can be studied directly or interpreted in terms of continued fraction expansions as above (following the point of view described in \cite{Popescu-Pampu2007}).
	The equivalence with the properties $(iii)$ and $(iv)$, which will not be used in the paper, are simple verifications left to the reader.
\end{proof}

The following result follows immediately from the condition $(v)$ above and Bezout Theorem, but any reader who's passionate about continued fractions can also obtain it from condition $(ii)$ as a simple exercise.

\begin{corollary}\label{corollary_existence_regular_fractional_annuli_coprime_multiplicities}
	Let $m,m'\in\N_{>0}$ such that $\gcd(m,m')=1$.
	Then there exists a regular fractional annulus whose ends have multiplicities $m$ and $m'$.
\end{corollary}

\begin{remarks}
\begin{enumerate}
	\item It is clear from the proposition that the regularity of fractional annuli does not descend nor it ascends with respect to base changes, even for base changes of degree prime with the multiplicities at the boundary.
	
	\item The proposition yields a simple effective procedure to resolve a fractional annulus, by adding the type 2 points of the skeleton that have multiplicity lower than or equal to the multiplicity of one of its boundary points.
	Because of the fact that regularity is not base-sensitive, and the fact that the canonical model of a semi-affinoid $X$ has a natural induced structure of formal scheme over $\fraks(X)^\circ$ (see \ref{def_Us}), the algorithm above works for generalized fractional annuli as well.
	In particular, by counting the number of downward blowups in this resolution process, we can observe that if $X\subset \A^{1,\an}_K$ is a fractional annulus of the form $\{|\pi|^{\sfrac{a}{b}}<|T(x)|<1\}$, with ${\sfrac{a}{b}}=[a_0;a_1,\ldots,a_{n}]$, then the smallest number of regular annuli in which $X$ can be broken is $\sum_{i\geqslant0}a_{2i}$.

	\item\label{remark_regular_with_rational_point_is_disk} A semi-affinoid $K$-curve $X$ that is regular and has a rational point over $K$ is necessarily a disc (see for example the implication (2)$\implies$(1) of \cite[Proposition 8.9]{Nicaise2009}).
	This fact and the examples above prompt us to ask whether it is true that a semi-affinoid curve that is not a disc is regular if and only if it has no $K(\pi^{\sfrac{1}{d}})$-rational point for every $d<\sum_{x\in\partial X}m(x)$.
\end{enumerate}
\end{remarks}

\bibliographystyle{alpha}                              
\bibliography{morph.bib}

\vfill

\end{document}